\documentclass[11pt]{amsart}

\usepackage[margin=2.5cm]{geometry}
\usepackage{color}
\usepackage{mathrsfs}
\usepackage{mathtools}
\usepackage{amsmath}
\usepackage{amssymb}
\usepackage{enumitem}
\usepackage{bbm}
\usepackage{esint}
\usepackage{nicefrac}
\numberwithin{equation}{section}
\usepackage[colorlinks,citecolor=green,linkcolor=red]{hyperref}

\usepackage[latin1]{inputenc}
\usepackage{tcolorbox}

\usepackage{todonotes}

\newtheorem{theorem}{Theorem}[section]

\newtheorem{lemma}[theorem]{Lemma}
\newtheorem{proposition}[theorem]{Proposition}

\theoremstyle{definition}
\newtheorem{definition}[theorem]{Definition}

\newtheorem{remark}[theorem]{Remark}

\newcommand{\N}{\mathbb{N}}

\newcommand{\R}{\mathbb{R}}
\renewcommand{\S}{\mathbb{S}}
\newcommand{\Z}{{\rm Z}}

\newcommand{\sfd}{{\sf d}}

\newcommand{\rr}{\mathbb R}
\newcommand{\restr}[1]{\lower3pt\hbox{$|_{#1}$}}
\newcommand{\down}{\downarrow}              

\newcommand{\eps}{\varepsilon}  
\newcommand{\nchi}{{\raise.3ex\hbox{$\chi$}}}
\newcommand{\weakto}{\rightharpoonup}

\newcommand{\fr}{\hfill$\blacksquare$}  

\newcommand{\HH}{\mathcal{H}} 

\newcommand{\LIP}{\mathrm{LIP}}
\newcommand{\Lip}{\mathrm{Lip}}
\newcommand{\lip}{\mathrm{lip}}

\newcommand{\diam}{\mathrm{diam}}

\newcommand{\RCD}{{\sf RCD}}

\newcommand{\mm}{\mathfrak m}

\renewcommand{\limsup}{\varlimsup}
\renewcommand{\liminf}{\varliminf}
\renewcommand{\d}{{\rm d}}
\newcommand{\X}{{\rm X}}

\newcommand{\Xdm}{(\X,\sfd,\mm)}
\newcommand{\rmCh}{{\rm Ch}}

\newcommand{\supp}{{\rm supp}}

\renewcommand{\phi}{\varphi}
\newcommand{\avr}{{\sf AVR}}

\title[]{Generalized existence of extremizers for the sharp $p$-Sobolev inequality on Riemannian manifolds with nonnegative curvature}

\author[]{Francesco Nobili} 
\address{Universit\'a di Pisa, Dipartimento di Matematica, Largo Bruno Pontecorvo 5,
56127 Pisa, Italy}
\email{\url{francesco.nobili@dm.unipi.it}}

\author[]{Ivan Yuri Violo}
\address{Universit\'a di Pisa, Dipartimento di Matematica, Largo Bruno Pontecorvo 5,
56127 Pisa, Italy}
\email{\url{ivanyuri.violo@dm.unipi.it}}

\begin{document}
\begin{abstract}
We study the generalized existence of extremizers for the sharp $p$-Sobolev inequality on noncompact Riemannian manifolds in connection with nonnegative curvature and Euclidean volume growth assumptions. Assuming a nonnegative Ricci curvature lower bound, we show that almost extremal functions are close in gradient norm to radial Euclidean bubbles. In the case of nonnegative sectional curvature lower bounds, we additionally deduce that vanishing is the only possible behavior, in the sense that almost extremal functions are almost zero globally.  Our arguments rely on nonsmooth concentration compactness methods and Mosco-convergence results for the Cheeger energy on noncompact varying spaces, generalized to every exponent $p\in (1,\infty)$.
\end{abstract}%
\maketitle
 \allowdisplaybreaks
\setcounter{tocdepth}{2}
\tableofcontents

\section{Introduction}
In this note we study the generalized existence of extremal functions for Sobolev inequalities on $d$-dimensional Riemannian manifolds $(M,g)$, $d \ge 2$, satisfying the following assumptions
\begin{equation}
    {\sf Ric}_g \ge 0, \qquad {\sf AVR }(M)\coloneqq \lim_{R\to\infty}\frac{{\rm Vol}_g(B_{R}(x))}{\omega_d R^d }>0, \label{eq:AVR}
\end{equation}
for $x \in M$. The constant ${\sf AVR }(M)$ is called the asymptotic {ratio} growth and, thanks to the Bishop-Gromov monotonicity, it holds that ${\sf AVR}(M)\in[0,1]$ and the limit exists and it is independent of $x$. The class \eqref{eq:AVR} is rich and contains many examples besides the Euclidean space $\R^d$ such as: Ricci flat asymptotical locally Euclidean manifolds and, in dimension four, gravitational instantons (\cite{Hawking77}). We refer to \cite{EguchiHanson79} for a concrete example of the so-called Eguchi-Hanson metric. Moreover, it was shown in \cite{Menguy00} that there are infinite topological types. Besides, spaces satisfying \eqref{eq:AVR} constitute an important class in geometric analysis and further examples are also weighted convex cones (see \cite{CabreRosOtonSerra16} and references therein) and, as we will see later, cones arising as limits of manifolds with Ricci curvature lower bounds.

\medskip

Starting from the works \cite{Ledoux00,Xia01}, it became clear that \eqref{eq:AVR} is a natural setup for the study of Sobolev inequalities of Euclidean type. Indeed, and more recently, it was shown in \cite{BaloghKristaly21} (see also \cite{Kristaly23} revisiting \cite{C-ENV04}) the validity for every $p\in(1,d)$ of the following
\begin{equation}
\| u\|_{L^{p^*}(M)} \le {\sf AVR }(M)^{-\frac{1}{d}}S_{d,p}\| \nabla u\|_{L^{p}(M)},\qquad \forall u \in \dot W^{1,p}(M).
\label{eq:AVR sob intro}
\end{equation}
on manifolds satisfying \eqref{eq:AVR}. Here we denoted $p^*\coloneqq pd/(d-p)$ the Sobolev conjugate exponent, by $S_{d,p}>0$ the sharp Euclidean Sobolev constant explicitly computed by \cite{Aubin76-2,Talenti76} (see \eqref{eq:Sobolev constant} for the precise value) and by $\dot W^{1,p}(M) \coloneqq \{u \in L^{p^*}(M) \colon |\nabla u|\in L^p(M)\}$ the homogeneous Sobolev space. Inequality \eqref{eq:AVR sob intro} is sharp (\cite{BaloghKristaly21}), and rigid as it was recently proved in    \cite{NobiliViolo24_PZ} (see \cite{CatinoMonticelli22,NobiliViolo24} for previous results with $p=2$). By rigid, we mean that equality holds in \eqref{eq:AVR sob intro} for some $0\neq u \in \dot W^{1,p}(M)$ if and only if $M$ is isometric to $\R^d$. Therefore, by the characterization of equality in the sharp Euclidean Sobolev inequality \cite{Aubin76-2,Talenti76} we deduce that $u$ has the following form
\begin{equation}
    u_{a,b,z_0}= \frac{a}{\Big(1+(b\sfd_g(\cdot,z_0)^{\frac{p}{p-1}}\Big)^{\frac{d-p}{p}}},
\label{eq:bubble}
\end{equation}
for some $a\in\R,b>0,z_0 \in M$. The above functions are usually called \emph{Euclidean bubbles} due to their radial shape. Finally, we recall that ${\sf AVR}(M)=1$ occurs if and only if $M$ is isometric to $\R^d$ (see \cite{Colding97}). Hence, a direct corollary of this rigidity principle is that, if ${\sf AVR}(M)\in(0,1)$, then there are no nonzero extremal functions for \eqref{eq:AVR sob intro}. We refer to \cite{Nobili24_overview} for an overview of these results and more references.
\medskip 

\noindent\textbf{Main results}. The Sobolev inequality \eqref{eq:AVR sob intro}, even though it does not admit nonzero extremizers, is sharp on every Riemannian manifold as in \eqref{eq:AVR}. Therefore, by definition, it is always possible to consider extremizing sequences:
\[
    0\neq u_n \in \dot W^{1,p}(M)\quad \text{so that}\quad \frac{\|u_n\|_{L^{p^*}(M)}}{\|\nabla u_n\|_{L^p(M)}} \to {\sf AVR}(M)^{-\frac 1p}S_{d,p},
\]
as $n\uparrow\infty$. In our previous work \cite{NobiliViolo24}, which is limited to the exponent $p=2$, we proved that there are $a_n\in\R,b_n>0,z_n$ so that
\[
    \lim_{n\uparrow\infty}\frac{\| \nabla (u_n -u_{a_n,b_n,z_n})\|_{L^2(M)} }{\|\nabla u_n\|_{L^2(M)}} \to 0.
\]
This means that the family of Euclidean bubbles actually completely captures the behaviors of extremizing sequences. However, since $u_n$ cannot converge in $\dot W^{1,2}(M)$ to some Euclidean bubble (unless $M$ is isometric to $\R^d$), the parameters $a_n,b_n,z_n$ are necessarily so that $u_{a_n,b_n,z_n}$ (renormalized) is either vanishing or is lacking compactness in the $\dot W^{1,2}(M)$ topology. 

\medskip 

In this note, we shall pursue the following two goals:
\begin{itemize}
   \item  extend the results of \cite{NobiliViolo24} to any exponent $p\neq 2$;
   \item relate geometric and curvature assumptions to a finer study of the behaviors of extremizing sequences.
\end{itemize}
We next present our main results and explain accurately, after the statements, how the above goals are achieved. 
\begin{theorem}\label{thm:qualitative Sob Manifold}
For all $\eps>0, V\in(0,1), d>1$ and $p \in (1,d)$, there exists $\delta\coloneqq \delta(\eps,p,d,V)>0$ such that the following holds.
Let $(M,g)$ be a noncompact $d$-dimensional Riemannian manifold with ${\sf Ric}_g \ge 0$ and ${\sf AVR}(M) \in( V,1]$ and let $0\neq u\in \dot W^{1,p}(M)$ be satisfying
\[
  \frac{\| u\|_{L^{p^*}(M)}}{\| \nabla u\|_{L^p(M)}} > {\sf AVR}(M)^{-\frac 1d} S_{d,p}- \delta.
\]
Then, there are $a\in\R,b>0$ and $z_0 \in M$ so that 
\begin{equation}
    \frac{\| \nabla( u - u_{a,b,z_0})\|_{L^{p}(M)}}{\| \nabla u\|_{L^{p}(M)}} \le \eps.\label{eq:gradient p stability}
\end{equation}
\end{theorem}
The above result fully extends \cite[Theorem 1.4]{NobiliViolo24} to any exponent $p\neq 2$.  The strategy boils down to generalized concentration compactness methods in the spirit of \cite{Lions84,Lions85} exploiting stability properties of non-smooth $\RCD(0,N)$ spaces (see Section \ref{sec:RCD}). In particular, Theorem \ref{thm:qualitative Sob Manifold} will be deduced from a more general analysis carried in Theorem \ref{thm:qualitative Sob} on $\RCD$ spaces covering, thus, also weighted Riemannian manifolds with nonnegative Barky-\'Emery Ricci curvature. Specifically, the two main ingredients are:
\begin{enumerate}[label=\alph*)]
    \item A general concentration compactness principle for $W^{1,p}$-functions along a sequence of metric measure spaces, that we develop in this note and extending the one of \cite{NobiliViolo21,NobiliViolo24} which was limited to $p=2.$
    \item The characterization of equality cases of $p$-Sobolev inequalities on nonsmooth spaces which we proved in \cite[ii) in Theorem 1.5]{NobiliViolo24_PZ}.
\end{enumerate}
To deal with a) we need to develop some technical tools about $W^{1,p}$-convergence on varying spaces, which we believe to be of independent interest (see Section \ref{sec:Mosco}). Mainly we  obtain the Mosco-convergence for the $p$-Cheeger energies on varying $\RCD(K,N)$ spaces. This extends the work \cite{AmbrosioHonda17}, for $N<\infty,$  by removing assumptions of finite reference measure or the presence of a common isoperimetric profile. Furthermore, we prove the linearity of the $W^{1,p}$-strong convergence and  the strong $L^p$-convergence of gradients.  To our best knowledge, these results were not known besides for the exponent $p=2$.

\medskip 

We next present our second main result where we further assume nonnegative sectional curvature. When the manifold is not isometric to $\R^d$, this more stringent assumption effectively narrows the range of possible behaviors for minimizing sequences.
\begin{theorem}\label{thm: stability Sect manifold}
    Let $(M,g)$ be a noncompact $d$-dimensional Riemannian manifold with ${\sf Sect}_g\ge 0$ and ${\sf AVR}(M) \in(0,1)$. Then, for every $\eps>0$ there exists $\delta = \delta(M,\eps)>0$ so that the following holds: if $0\neq u\in \dot W^{1,p}(M)$ satisfies
    \[
        \frac{\| u\|_{L^{p^*}(M)}}{\| \nabla u\|_{L^p(M)}} > {\sf AVR}(M)^{-\frac 1d} S_{d,p}- \delta,
    \]
    then, there are $a\in\R,b>0$ and $z_0 \in M$ so that
    \begin{equation}\label{eq:stab sect}
        \frac{\| \nabla( u - u_{a,b,z_0})\|_{L^{p}(M)}}{\| \nabla u\|_{L^{p}(M)}} \le \eps,\qquad\text{and} \qquad    |u_{a,b,z_0}|\le \|u\|_{L^{p^*}(M)}\eps, \quad \text{in $M$}
    \end{equation}
    (or equivalently $b<\eps).$
    Furthermore, writing $M=\rr^k\times N$ for some $0\le k<d$ and some $(d-k)$-dimensional Riemannian manifold $(N,h)$ that does not split isometrically any line, we can take
    \begin{equation} \label{eq:stab sect center}
        z_0\in \rr^k\times \{y_0\},
    \end{equation}
    for any fixed $y_0\in N$ (with $\delta$ depending also on $y_0$). 
\end{theorem}
{ Some comments on the above statement are in order:
\begin{enumerate}[label=\roman*)]
    \item The second inequality in \eqref{eq:stab sect} is saying that a function which is almost extremal for the Sobolev inequality in $M$, must be almost zero in the sense that it is  $\dot W^{1,p}$-close to a bubble which is close to zero   uniformly in $M.$ In other words minimizing sequences must be very diffused on the whole manifold.
    \item The second part of the theorem  instead says roughly that almost extremal functions do not escape at infinity.  More precisely \eqref{eq:stab sect center} says that an almost extremal function must be close to a bubble that can be centred at any chosen point $z_0$, up to isometries of $M$ and taking $\delta$ sufficiently small. In other words, extremizing sequences diffuse faster than the rate at which they might escape to infinity.
    \item The exact same result of Theorem \ref{thm: stability Sect manifold} holds for $d$-dimensional convex subsets of $\rr^n$  with positive asymptotic volume ratio, which are not cones (the sharp Sobolev inequality on noncompact convex subsets of $\rr^n$ is a consequence of \cite[Theorem 1.13]{NobiliViolo21}). In fact, we prove the result for the more general class of Alexandrov spaces with nonnegative sectional curvature (that are not cones), see Theorem \ref{thm:stability CBB}.
    \item  It is worth to observe that \eqref{eq:stab sect center} does not follow from \eqref{eq:stab sect}. Indeed  $\|u_{a,b,z_0}-u_{a,b,z_1}\|_{L^{p^*}(M)}\ge \|u_{a,b,z_0}\|_{L^{p^*}(M)}/2$, no matter what $a$ and $b$ are, provided $z_0,z_1$ are sufficiently far apart.
    \item The conclusion \eqref{eq:stab sect} holds under a weaker assumption on the volume of small balls, see Theorem \ref{thm:refined stability}.
    \item The second part of the statement of Theorem \ref{thm: stability Sect manifold} does not hold if we assume only non-negative Ricci curvature, see Remark \ref{rmk:counter}.
\end{enumerate}
}
The proof of Theorem \ref{thm: stability Sect manifold} rests on the rigidity properties of blow-downs (also called asymptotic cones) for spaces with non-negative sectional curvature.
Similar ideas were recently employed to prove  existence results for isoperimetric sets on noncompact manifolds \cite{AntonelliBrueFogagnoloPozzetta22} (see also \cite{AntonelliPozzetta23}).
\section{Preliminaries}
We start by introducing some relevant notation. For every $N>1$ and $p\in(1,\infty)$ we denote
\begin{equation}
  S_{N,p} \coloneqq \frac 1N\left(\frac{N(p-1)}{N-p}\right)^{\frac{p-1}p}\left(\frac{\Gamma(N+1)}{N\omega_N\Gamma(N/p)\Gamma (N+1-d/p)}\right)^{\frac 1N}, \label{eq:Sobolev constant}  
\end{equation}
where $\omega_N \coloneqq \pi^{N/2} / \Gamma(N/2+1)$ and $\Gamma(\cdot)$ is the Gamma function.

A metric measure space is a triple  $(\X,\sfd,\mm)$ where $(\X,\sfd)$ is a complete and separable metric space and $\mm$ is a non-negative, non-zero and boundedly finite Borel measure. By $C(\X),C_b(\X),C_{bs}(\X)$ we denote respectively the space of continuous functions, continuous and bounded functions and continuous and boundedly supported functions on $\X$. By $\Lip(\X),\Lip_{bs}(\X)$, we denote respectively the collection of Lipschitz functions and boundedly supported Lipschitz functions and by $\lip(u)$ the local Lipschitz constant of $u\colon \X\to \R$. For all $p \in (1,\infty)$, we denote by $L^p(\mm),L^p_{loc}(\mm)$ respectively the space of $p$-integrable functions and $p$-integrable functions on a neighborhood of every point (up to $\mm$-a.e.\ equality relation) on $\X$.

\subsection{Calculus on ${\sf RCD}$ spaces}\label{sec:RCD}
We define the $p$-\emph{Cheeger} energy by
\[
\rmCh_p(u)\coloneqq \inf\left\{ \liminf_{n\to\infty}\int \lip(u_n)^p\,\d \mm \colon (u_n)\subseteq \Lip(\X), \, u_n \to u \text{ in }L^p(\mm)\right\}.
\]
The Sobolev space $W^{1,p}(\X)$ is defined as the collection of $u \in L^p(\mm)$ so that $\rmCh_p(u)<\infty$ equipped with the usual norm $\|u\|_{W^{1,p(\X)}}^p \coloneqq \|u\|_{L^p(\mm)}^p + \rmCh_p(u)$. We refer to \cite{Cheeger00,Shan00} for a general introduction while here we follow the equivalent axiomatization given by \cite{AmbrosioGigliSavare11-3}. Recall that we have the representation
\[
\rmCh_p(u) = \int |\nabla u|^p\,\d \mm,
\]
for a suitable function $|\nabla u| \in L^p(\mm)$ called minimal $p$-weak upper gradient. Thanks to locality \cite{AmbrosioGigliSavare11-3} of minimal $p$-weak upper gradients, we recall the space $W^{1,p}_{loc}(\X)$ as the subset of $u \in L^p_{loc}(\mm)$ so that $\eta u \in W^{1,p}(\X)$ for all $\eta \in \Lip_{bs}(\X)$. By slight abuse of notation, we shall write $\|\nabla u\|_{L^p(\mm)}$ in place of $\| |\nabla u|\|_{L^p(\mm)}$. We also do not insist on the dependence of $|\nabla u|$ on the exponent $p$ (see, e.g., \cite{DiMarinoSpeight13}), as we shall only deal with settings where this does not occur, see \cite{Cheeger00,GigliHan14,GigliNobili21}.

We also recall the notion of functions of locally bounded variation following \cite{Miranda03,AmbrosioDiMarino14}. If $\varnothing \neq U\subset\X$ is open and $u\in L^1_{loc}(\mm)$, we define
\[
    |D u|(U) \coloneqq \inf \Big\{\liminf_{n\uparrow\infty} \int \lip\, u_n\, \d \mm \colon (u_n)\subset \Lip_{loc}(U),\, u_n \to u \text{ in }L^1_{loc}(U)\Big\}.
\]
{If the above is locally finite, } it can be shown that it extends to a nonnegative Borel measure to the whole sigma-algebra of Borel sets (\cite{AmbrosioDiMarino14}). We then say that $u \in BV_{loc}(\X)$ provided $|Du|$ is finite on a neighborhood of every point. We simply say that $u$ is a function of bounded variation, writing $u \in BV(\X)$, provided $u\in L^1(\mm)$ and $|Du|(\X)<\infty$. We also refer to \cite{DiMarinoPhD,Martio16-2,NobiliPasqualettoSchultz21,BrenaNobiliPasqualetto22} for other equivalent approaches.

In this note we are interested in Sobolev inequalities on spaces with synthetic Ricci curvature lower bounds. We assume the reader to be familiar with the theory and concepts of ${\sf RCD}$-spaces. In the following parts, we shall limit ourselves to recalling only the relevant properties. We refer, for general introductions and the relevant references to the surveys \cite{Villani2016,AmbICM,Gigli23_working,Sturm24_Survey}.

If $\Xdm$ is an ${\sf RCD}(0,N)$ space  for some  $N>1$ we will use several times the following Bishop-Gromov monotonicity (see \cite{Sturm06I,Sturm06II}): for all $x\in\X$ we have that\[
r\mapsto \frac{\mm(B_r(x))}{\omega_Nr^N},\qquad\text{is non-increasing}.
\]
In particular, the following limit is well-defined and independent on $x$
\[
{\sf AVR}(\X) \coloneqq \lim_{r\to \infty} \frac{\mm(B_r(x))}{\omega_Nr^N} \in [0,\infty).
\]
We next state a useful principle for Sobolev functions and functions of bounded variations.
\begin{proposition}\label{prop:sob small support}
    For all constants $\eps>0,K\in\R,N\in(1,\infty),p\in[1,N)$ and $R_0>0$ there exists $\delta \coloneqq \delta(\eps,K,N,p,R_0)>0$ so that the following holds: let $\Xdm$ be an ${\sf RCD}(K,N)$ space, $x \in \X$ and suppose that $u \in W^{1,p}(\X)$ if $p\neq 1$ or $u \in BV(\X)$ if $p=1$ satisfies for some $0<R\le R_0$
    \begin{equation}
    \supp(u)\subset B_R(x),\qquad \frac{\mm(\supp(u))}{\mm(B_R(x))} \le \delta.
    \label{eq:delta support}
    \end{equation}
    Then, it holds
    \[
        \int |u|^p\,\d \mm \le  \eps\cdot \begin{cases}
        \, \int |\nabla u|^p\,\d \mm,&\text{if }p\in(1,N),\\
        \, |Du|(\X),            &\text{if }p=1.
        \end{cases}
    \]
\end{proposition}
\begin{proof}
    We only prove the case $p>1$, the case $p=1$ being the same. Denote by $p^*=pN/(N-p)>p$ and, for all $\eps>0$, notice by interpolation and Young inequality that
    \[
    \left(\fint_{B_R(x)}|u|^p\,\d\mm\right)^\frac1p \le \eps\left( \fint_{B_R(x)}|u|^{p^*}\,\d\mm\right)^{\frac{1}{p^*}} + \eps^{-\frac{N(p-1)}{p}}\fint_{B_r(x)}|u|\,\d\mm. 
    \]
    Thus, provided $\delta^{1-1/p} \le \frac 12 \eps^{\frac{N(p-1)}{p}}$, by H\"older inequality and the assumptions \eqref{eq:delta support} we get
    \[
     \left(\fint_{B_R(x)}|u|^p\,\d\mm\right)^\frac1p  \le 2\eps \left( \fint_{B_R(x)}|u|^{p^*}\,\d\mm\right)^{\frac{1}{p^*}}.
    \]
    In particular, by the triangular inequality and the local $(p^*,p)$-Sobolev inequality in this setting (see, e.g.\ \cite[Theorem 5.1]{HajlaszKoskela00}), we deduce
    \[
    \left( \fint_{B_R(x)}|u|^{p^*}\,\d\mm\right)^{\frac{1}{p^*}} - \fint_{B_R(x)}|u|\,\d\mm \le  \left( \fint_{B_R(x)}\Big|u - \fint_{B_R(x)}u\,\d \mm\Big|^{p^*} \,\d\mm\right)^{\frac{1}{p^*}} \le C \left(\fint_{B_R(x)}|\nabla u|^p\,\d \mm\right)^\frac 1p,
    \]
    for some constant $C\coloneqq C(K,N,R_0)>0$. Combining everything and using again H\"older inequality on the term $\fint_{B_R(x)}|u|\,\d\mm $, the proof is concluded.
\end{proof}

We isolate here the following technical density bound that will be needed in the proof of Theorem \ref{thm:CC extremal} in the collapsed case. The proof is identical to \cite[Lemma 6.1]{NobiliViolo24}, there for $p=2$, and it is omitted.
\begin{lemma}[Density bound from reverse Sobolev inequality]\label{lem:density bound}
For every $N\in(1,\infty),p\in(1,N)$ and $K \in \R$ there are constants $\lambda_{N,{K},p}\in(0,1),\, r_{{K^-},N,p}>0$  (with $r_{0,N,p}=+\infty$), $C_{N,K,p}>0$  such that the following holds. Set $p^*=pN/(N-p)$ and let $\Xdm$ be an $\RCD(K,N)$ space and let $u \in W^{1,p}_{loc}(\X)\cap L^{p^*}(\mm)$ be non-constant satisfying 
    \begin{equation}\label{eq:reverse sobolev lemma}
           \|u\|_{L^{p^*}(\mm)}^p\ge A\|\nabla u\|_{L^p(\mm)}^p,
    \end{equation}
    for some $A>0.$ Assume also that for some $\eta \in(0,\lambda_{N,K,p})$, $\rho \in (0,r_{{K^-},N,p}\wedge \frac{\lambda_{N,K,p}}8 \diam(\X))$ and $x\in\X$ it holds
    $$\|u\|^{p^*}_{L^{p^*}(B_{\rho}(x))}\ge (1-\eta) \|u\|_{L^{p^*}(\mm)}^{p^*}.$$
    Then
        \begin{equation}\label{eq:density bound}
        \frac{\mm(B_\rho(x))}{\rho ^N}\le \frac{C_{N,K,p}}{A^{N/p}}.
    \end{equation}
\end{lemma}
\subsection{Convergence and stability properties}
In this part, we recall compactness and stability properties of the ${\sf RCD}$-class and discuss notions of convergence of functions on varying base space.

 We recall first the notion of pointed measured Gromov Hausdorff convergence of metric measure spaces. This concept goes back to Gromov \cite{Gromov07}, while the following definition is not standard and is taken from  \cite{GigliMondinoSavare13}. However, in the case of finite dimensional ${\sf RCD}$-spaces, this notion coincides with previous ones considered in the literature (see again \cite{GigliMondinoSavare13}). 
 
 Set $\bar \N \coloneqq \N \cup\{\infty\}$. A \emph{pointed} metric measure space is a quadruple $(\X,\sfd,\mm,x)$ where $\Xdm$ is a metric measure space and {$x \in \supp(\mm)$}.
\begin{definition}
Let $(\X_n,\sfd_n,\mm_n,x_n)$ be pointed metric measure spaces for $n \in \bar \N$. We say that $(\X_n,\sfd_n,\mm_n,x_n)$ converges to $(\X_\infty,\sfd_\infty,\mm_\infty,x_\infty)$ in the pointed measured Gromov Hausdorff topology, provided there are a metric space $(\Z,\sfd)$ and isometric embeddings $\iota_n \colon \X_n \to \Z$ for all $n \in \bar \N$ satisfying
\[
    (\iota_n)_\sharp \mm_n \weakto (\iota_\infty)_\sharp\mm_\infty,\qquad \text{in duality with }C_{bs}(\Z),
\]
and $\iota_n(x_n)\to \iota_\infty(x_\infty)$ as $n\uparrow\infty$. The metric space $(\Z,\sfd)$ is called the \emph{realization of the convergence}. In this case, we shortly say that $\X_n$ pmGH-converges to $\X_\infty$ and write $\X_n \overset{pmGH}{\to}\X_\infty$. 
\end{definition}

The key results are then the pre-compactness and the stability properties of the ${\sf RCD}$-condition, referring to \cite{Gigli10,AmbrosioGigliSavare11-2,GigliMondinoSavare13} (also recall \cite{Sturm06I,Sturm06II,LV09}) and thanks to Gromov's precompactness \cite{Gromov07}). 
\begin{theorem}\label{thm:gromov}
    Let $(\X_n,\sfd_n,\mm_n,x_n)$ be pointed ${\sf RCD}(K_n,N_n)$ spaces for $n\in\N$ and for some $K_n \in \R,N_n \in [1,\infty)$ with $K_n \to K \in \R, N_n \to N \in [1,\infty)$. Suppose that $\mm_n(B_1(x_n)) \in (v,v^{-1})$ for some $v>0$ independent on $n$. Then, there exist a pointed ${\sf RCD}(K,N)$ space $(\X_\infty,\sfd_\infty,\mm_\infty,x_\infty)$ and a subsequence $(n_k)$ such that $\X_{n_k} \overset{pmGH}{\to}\X_\infty$ as $k\uparrow\infty$.
\end{theorem}
It is well known that, as a by-product of the above result, we have the existence of blow downs (or asymptotic cones) of a pointed ${\sf RCD}(0,N)$ metric measure space $(\X,\sfd,\mm,x)$ with ${\sf AVR}(\X)>0$. A blowdown is any pointed metric measure space $(Y,\rho,\mu,y)$  arising as a pmGH-limit of $(\X,\sigma\cdot\sfd, \sigma^N\mm,x)$ along a suitable subsequence $\sigma_n\down 0$, possibly depending on $x\in \X$.

Next, we recall some notions of convergence of functions along a pmGH-converging sequence, following \cite{Honda15,GigliMondinoSavare13,AmbrosioHonda17}and adopting the so-called \emph{extrinsic approach}, see \cite{GigliMondinoSavare13}.
\begin{definition}\label{def:lpconv}
 Let $(\X_n,\sfd_n,\mm_n,x_n)$ be pointed metric measure spaces for $n \in \bar \N$ and suppose that $\X_n \overset{pmGH}{\to}\X_\infty$ as $n\uparrow\infty$. Let $p \in (1,\infty)$ and fix a realization  of the convergence in $(\Z,\sfd)$. We say:
\begin{itemize}
\item[{\rm i)}] $f_n\in L^p(\mm_n)$ \emph{converges $L^p$-weak} to $f_\infty\in L^p(\mm_\infty)$, provided $\sup_{n \in \N}\|f_n\|_{L^p(\mm_n)}<\infty$ and $f_n\mm_n \weakto f_\infty\mm_\infty$ in duality with $C_{bs}(\Z)$;
\item[{\rm ii)}] $f_n\in L^p(\mm_n)$ \emph{converges $L^p$-strong} to $f_\infty\in L^p(\mm_\infty)$, provided it converges $L^p$-weak and $\limsup_n \|f_n\|_{L^p(\mm_n)} \le  \|f_\infty\|_{L^p(\mm_\infty)}$;
\item[{\rm iii)}] $f_n \in W^{1,p}(\X_n)$ \emph{converges $W^{1,p}$-weak} to $f_\infty \in{W^{1,p}(\X_\infty)}$ provided it converges $L^p$-weak and $\sup_{n \in \N} \| \nabla f_n \|_{L^p(\mm_n)}<\infty$;
\item[{\rm iv)}] $f_n \in W^{1,p}(\X_n)$ \emph{converges $W^{1,p}$-strong} to $f_\infty \in W^{1,p}(\X_\infty)$ provided it converges $L^p$-strong and $\| \nabla f_n \|_{L^p(\mm_n)} \to \|\nabla f_\infty \|_{L^p(\mm_\infty)}$;
\item[{\rm v)}] $f_n\in L^p_{loc}(\mm_n)$ \emph{converges $L^p_{loc}$-strong} to $f_\infty\in L^p_{loc}(\mm_\infty)$, provided $\eta f_n$ converges $L^p$-strong to $\eta f_\infty$ for every $\eta \in C_{bs}(\Z)$.
\item[{\rm vi)}] {$f_n \in L^1(\mm_n)$ \emph{converges $L^1$-strong} to $f_\infty \in L^1(\mm_\infty)$, provided $\sigma \circ f_n \mm_n\weakto \sigma \circ f_\infty \mm_\infty$ in duality with $C_{bs}(\Z)$ and $\lim_{n\to\infty} \|f_n\|_{L^1(\mm_n)} = \|f_\infty \|_{L^1(\mm_\infty)}$, where $\sigma(z)=\text{sgn}(z)\sqrt{z}$.}
\end{itemize}
\end{definition}
We point out (see \cite{Honda15,GigliMondinoSavare13,AmbrosioHonda17}) that the strong notions of convergence are also metrizable. 
\subsection{Alexandrov spaces: asymptotic geometry}
We recall some useful results around ${\sf CBB}(0)$ spaces, namely metric spaces $(\X,\sfd)$ with nonnegative sectional curvature lower bounds in the sense of Alexandrov. We refer to \cite{BBI01,AKP} for detailed discussions and references and to the foundational works \cite{Alexandrov41} and \cite{BuGroPer92} of the Alexandrov geometry.

We first recall the concept of triangle comparison. Given a {complete} geodesic metric space $(\X,\sfd)$, and three points $a,b,c\in\X$, then we consider three points in $\R^2$ (unique up to isometries) called \emph{comparison points}, $\bar a,\bar b,\bar c\in\R^2$ such that
\[
|\bar a - \bar b|=\sfd(a,b),\qquad\qquad |\bar b- \bar c|=\sfd(b,c),\qquad\qquad|\bar c-\bar a|=\sfd(c,a).
\]
A point $d\in\X$ is said to be intermediate between $b,c\in\X$ provided $\sfd(b,d)+\sfd(d,c)=\sfd(b,c)$ (this means that $d$ lies on a geodesic joining $b$ and $c$). The \emph{comparison point of $d$} is the unique (once $\bar a,\bar b,\bar c$ are fixed) point $\bar d\in\R^2$, such that $|\bar d-\bar b|=\sfd(d,b),$ and $|\bar d-\bar c|=\sfd(d,c)$.
\begin{definition}[${\sf CBB}(0)$ space]
    A metric space $(\X,\sfd)$ is a ${\sf CBB}(0)$-space, provided for every triple of points $a,b,c \in \X$ and for every intermediate point $d\in \X$ between $b,c$, it holds
    \[
        \sfd(d,a) \ge |\bar d - \bar a|.
    \]
\end{definition}
Several equivalent definitions, in terms of comparison angles or properties of the distance function along geodesics, can be given. We refer to \cite{BBI01} for a complete account and reference. 

It is well known that Alexandrov spaces have \emph{integer dimension} and have a well-behaved local and asymptotic geometry, see \cite{BBI01,AKP}. Given $N\in\N$, we say that $(\X,\sfd)$ is an $N$-dimensional ${\sf CBB}(0)$-space, provided it has Hausdorff dimension $N$. We denote by $\HH^N$ the Hausdorff measure in this case, built-in on top of the metric $\sfd$ with the usual construction. Related to this, we recall the compatibility result
\[
(\X,\sfd,\HH^N)\qquad \text{is an }{\sf RCD}(0,N)\text{-space},
\]
as outcome of \cite{Petrunin11,ZhangZhu10,GKKO} (in fact, it is also non-collapsed \cite{DePhilippisGigli15}). In the next result we report on the asymptotic geometry of Alexandrov spaces that turns out to be much better behaved as compared to that of ${\sf RCD}$ spaces.
\begin{theorem}\label{thm:asymptotic cone CBB}
    Let $(\X,\sfd)$ be a $N$-dimensional ${\sf CBB}(0)$-space for some $N\in\N$ with ${\sf AVR}(\X)>0$. Then, there is a unique blow down $(Y,\rho,\HH^N)$.  Moreover, $(\X,\sfd,\HH^N)$ splits isometrically a line if and only if $(Y,\rho,\HH^N)$ does.
\end{theorem}
A proof of the above can be found in \cite[Theorem 4.6]{AntonelliBrueFogagnoloPozzetta22} in the contexts of manifolds. We refer also to \cite[Theorem 2.11]{AntonelliPozzetta23} for the current setting and for further references.
\section{Convergence of functions on varying spaces: the case $p\neq 2$}\label{sec:Mosco}
The aim of this section is to develop technical convergence and compactness results around Sobolev functions on varying spaces for a general {integration} exponent $p \in (1,\infty)$. The case $p=2$ has been first analyzed in \cite{GigliMondinoSavare13} by studying the stability properties of heat flows and quantities related to the Entropy functional. Later, in \cite{AmbrosioHonda17} the case $p\neq 2$ has been faced by relying on self-improvement properties studied in \cite{Savare13}. Even though the analysis in \cite{AmbrosioHonda17} holds on possibly infinite dimensional spaces, some results there require a probability reference measure or the existence of a common isoperimetric profile along the sequence. 

In the setting of this note, we cannot assume finite reference measures as we are going to deal with sequences of noncompact ${\sf RCD}$-spaces with $\sigma$-finite reference measures. For finite dimensional spaces, we extend next the analysis of \cite{AmbrosioHonda17} completely dropping any further assumptions.

Throughout this section, we shall {consider fixed} a sequence $(\X_n,\sfd_n,\mm_n,x_n)$ of pointed $\RCD(K,N)$ spaces, $n \in \N\cup\{\infty\}$, for some $K\in\R, N \in (1,\infty)$ with $\X_n\overset{pmGH}{\to}\X_\infty$. A proper realization of this convergence $(\Z,\sfd)$ will also be fixed following the extrinsic approach \cite{GigliMondinoSavare13}.
\subsection{Mosco-convergence of Cheeger energies}\label{sec:compactness}
We study a {Rellich-type} compactness result allowing to extract $L^p$-strong converging subsequence from uniform $W^{1,p}$-bounds and equiboundedness of the supports.  Here we also cover the BV-case for the sake of generality.
\begin{proposition}\label{prop:rellich W1p}
    Let $p\in [1,\infty)$ and suppose that $u_n \in L^p(\mm_n)$ satisfies $\sup_n\|u_n\|_{L^p(\mm_n)}<\infty$ and $\supp(u_n)\subset B_R(x_n)$ for some $R>0$ independent on $n \in\N$. Furthermore, assume:    \begin{itemize}
        \item if $p\in(1,\infty)$ it holds $u_n \in W^{1,p}(\X_n)$ and $\sup_n \|\nabla u_n\|_{L^p(\mm)}<\infty$;
        \item if $p=1$ it holds $u_n \in BV(\X_n)$ and $\sup_n |Du_n|(\X_n)<\infty$.
    \end{itemize}   
    Then, up to passing to a subsequence, it holds that $u_n$  converges $L^p$-strong to some $u_\infty\in L^p(\mm_\infty)$. 
\end{proposition}
\begin{proof}
    The case $p=2$ follows by \cite[Theorem 6.3]{GigliMondinoSavare13}. For the case $p\neq 2$, we adapt the argument in \cite[Theorem 7.5]{AmbrosioHonda17}. Our goal is for all $\eps>0$ to write $u_n= g_n^\eps + h^\eps_n$ in such a way that $g^\eps_n$ converge $L^p$-strong to some $g^\eps$ and $\|h_n^\eps\|_{L^p(\mm_n)}\le \eps$. This would be sufficient to conclude, since $\|g^\eps-g^{\eps'}\|_{L^p(\mm_\infty)} = \lim_{n\uparrow\infty}\|g_n^\eps-g^{\eps'}_n\|_{L^p(\mm_n)} \le \eps +\eps'$, hence by taking a sequence $\eps_i \downarrow 0$ fast enough, we have that $g^{\eps^i}$ is $L^p(\mm_\infty)$-Cauchy and it converges strongly in $L^p(\mm_\infty)$ to some $g$. Then, a diagonal argument would give that $g^{\eps_n}_n$ converges along a suitable subsequence in $L^p$-strong to $g$. Since $\|u_n - g^{\eps_n}_n\|_{L^p(\mm_n)} = \| h^{\eps_n}_n\|_{L^p(\mm_n)}\le \eps_n$ by construction, this implies that also $u_n$ converges $L^p$-strong to $g$. Thus, setting $u_\infty \coloneqq g$ gives the conclusion.

    To produce the above decomposition of $u_n$, we proceed as follows. By pmGH-convergence $\liminf_n \mm_n(B_R(x_n))\ge \mm_\infty(B_R(x_\infty))\eqqcolon v>0.$ By Markov inequality and the assumption on the uniform $L^p$-boundedness of $u_n$, we have that for all $\delta>0$ there exists $M\coloneqq M(\delta)>0$ independent of $n\in\N $ so that $\mm_n\{|u_n|>M\}\le \delta v/2$. Set $g_n \coloneqq g_n(\delta) \coloneqq (-M)\wedge u_n \vee M$ and $h_n\coloneqq h_n(\delta)\coloneqq u_n-g_n$. In particular, $\supp (h_n)\subset \{ |u_n|>M\}$ and so $\mm_n(\supp(h_n))\le \delta  \mm_n(B_R(x_n))$ for all $n$ big enough. Therefore, by applying Proposition \ref{prop:sob small support}, and thanks to the assumption $\sup_n\|u_n\|_{W^{1,p}(\X_n)}<\infty$ (resp. $\sup_n \|u_n\|_{L^1(\mm_n)}+|Du_n|(\X_n)<\infty)$, we obtain that $\|h_n\|_{L^p(\mm_n)}<\eps$, provided $\delta$ is chosen small enough. From here, we set $g_n^\eps \coloneqq g_n(\delta)$ and $h^\eps_n\coloneqq h_n(\delta)$.

    We now distinguish the case $p>2$ and $p<2$, the first being simpler. If $p>2$, then also $\sup_n\|g^\eps_n\|_{W^{1,2}(\X_n)}<\infty$ and so by \cite[Theorem 6.3]{GigliMondinoSavare13} we deduce that $g_n^\eps$ converges $L^2$-strong to some function $g^\eps$. Since the sequence is equi-bounded, then by \cite[(e) in Proposition 3.3]{AmbrosioHonda17} we also deduce that $g_n^\eps$ converge $L^p$-strong to $g^\eps$.

    It remains the case where $p<2$. In this case, for $t>0$ we consider instead the sequence $h_t^ng^\eps_n$, where $t\mapsto h^n_tf$ denote the heat flow evolution for the $2$-Cheeger energy on $\X_n$ starting at $f \in L^p(\mm_n)$, see e.g.\ \cite{GigliPasqualetto20book}. By the $L^\infty$-to-Lipschitz regularization property (see, e.g., \cite[Proposition 6.1.6]{GigliPasqualetto20book}) and the Sobolev-to-Lipschitz property \cite{Gigli13_splitting}, we deduce that $h_t^n g_n^\eps$ are equi-Lipschitz, since $g_n^\eps$ are equi-bounded (in $n\in\N$). Fix a cut-off $\eta \in \Lip_{bs}(\Z)$ with $\eta\equiv 1$ on $B_{R+1}(x_\infty)$ and $|\eta|\le 1$. Then, $\sup_n\| \eta h_t^ng_n^\eps\|_{W^{1,2}(\X_n)}<\infty$ and so up to subsequence we have that $\eta ( h_t^ng_n^\eps ) $ converges $L^2$-strong to some function $g^\eps$. Since $g^\eps_n$ are all supported, for $n$ large enough, in $B_{R+1}(x_\infty)$, we have
    \[
    \| g^\eps_n - \eta h_t^ng_n^\eps\|_{L^2(\mm_n)} = \|\eta g^\eps_n -\eta h_t^ng_n^\eps\|_{L^2(\mm_n)}\le \| g^\eps_n - 
    h_t^ng_n^\eps\|_{L^2(\mm_n)}.
    \]
    By stability properties of the heat flow (c.f. \cite[Theorem 6.3]{GigliMondinoSavare13}), the last term goes to zero as $t\to \infty$ uniformly on $n$. Hence, also $g^\eps_n$ converges $L^2$-strong to $g^\eps$, by metrizability of $L^2$-strong convergence with varying base space. Again, this upgrades to $L^p$-strong convergence being the supports equibounded.
\end{proof}
{ The above result for $p>1$ has recently appeared in the independent work \cite[Theorem 6.14]{Wu25}, under more general assumptions using a different method.}

We next derive a general weak lower semicontinuity result on open sets.
 \begin{proposition}[Lower semicontinuity on open sets]\label{prop:lsc on open}
   Let $p\in(1,\infty)$ and suppose that $u_n \in W^{1,p}(\mm_n)$ converges $L^p$-weak to some $u_\infty \in L^p(\mm_\infty)$ with $\sup_{n\in\N} \|u_n\|_{W^{1,p}(\X_n)}<\infty$. Then, $u_\infty \in W^{1,p}(\X_\infty)$ and for every $A\subset \Z$ open, we have
   \begin{equation}
        \int_A |\nabla u_\infty|^p\,\d \mm_\infty \le \liminf_{n\uparrow\infty}\int_A |\nabla u_n|^p\,\d \mm_n.
   \label{eq:lsc chp on open}
   \end{equation}
    Similarly, suppose that $u_n \in BV(\X_n)$ {converges} in $L^1$-weak to $u_\infty \in L^1(\mm_\infty)$ and $\sup_n |Du_n|(\X_n)<\infty$. Then, $u_\infty \in BV(\X_\infty)$ and for every $A\subset \Z$ open we have
    \begin{equation}
        |Du_\infty|(A) \le \liminf_{n\uparrow\infty}|Du_n|(A). \label{eq:lsc Du on open}
   \end{equation}
\end{proposition}
\begin{proof}
    We subdivide the proof into two different steps handling the Sobolev and BV case at the same time.
    
    \noindent\textsc{Membership: $u_\infty \in W^{1,p}$/$BV$}. Let us first show that $u_\infty \in W^{1,p}(\X_\infty)$ (resp.\ $u_\infty \in BV(\X_\infty)$). Notice that this conclusion is non-trivial since $u_n$ only converges $L^p$-weak to $u_\infty$, hence \cite[Theorem 8.1]{AmbrosioHonda17} does not apply (while, for $p=2$, this is known \cite[ii) in Theorem 6.3]{GigliMondinoSavare13}). Consider $\eta \in \Lip_{bs}(\Z)$ and a suitable $R>0$ so that $\supp(\eta)\subset B_R(x_\infty)$. Then, since $x_n \to x_\infty$ in $\Z$, we have up to possibly discarding a finite number of indices that $ \supp(\eta u_n )\subset B_{2R}(x_n)$. Therefore, again up to a further subsequence, the compactness result in Theorem \ref{prop:rellich W1p} applies giving that $\eta u_n$ converges $L^p$-strong to some function $v$. However, $u_n$ is assumed to be $L^p$-weak converging to $u_\infty$, hence $\eta u_n$ is also $L^p$-weak converging to $\eta u_\infty$. By uniqueness of weak limits, we must have that $\eta u_n$ converges $L^p$-strong to $\eta u_\infty$ also along the original sequence.  If $p>1$ the Gamma-convergence result in \cite[Theorem 8.1]{AmbrosioHonda17} applies giving
    \[
        \rmCh_p^{1/p}(\eta u_\infty) \le \liminf_{n\uparrow\infty}\rmCh_p^{1/p}(\eta u_n)\le \liminf_{n\uparrow\infty}\, \Lip(\eta)\|u_n\|_{L^p(\mm_n)} + \|\eta\|_{L^\infty}\|\nabla u_n\|_{L^p(\mm_n)}  < \infty,
    \]
    by the Leibniz rule and the assumptions. Instead, if $p=1$ we can rely on \cite[Thorem 6.4]{AmbrosioHonda17} to deduce
    \[
        |D(\eta u_\infty)|(\X_\infty) \le \liminf_{n\uparrow\infty} |D(\eta u_n)|(\X_n) \le \liminf_{n\uparrow\infty} \Lip(\eta)\|u_n\|_{L^1(\mm_n)} + \|\eta\|_{L^\infty(\mm_n)}|Du_n|(\X_n)<\infty,
    \]
    again by the Leibniz rule for BV functions and the assumptions. All in all, by arbitrariness of $\eta$, we have just deduced that $u_\infty\in W^{1,p}_{loc}(\X_\infty)$ (resp. $u_\infty \in BV_{loc}(\X_\infty)$). Now, if we further choose $\eta$ to be $1$-Lipschitz with $|\eta|\le 1$ and such that $\eta=1$ in $B_{R-1}^Z(x_\infty)$, the above and locality in $W^{1,p}_{loc}(\X_\infty)$ guarantee that
    \[
     \|\nabla u_\infty\|_{L^p(B_{R-1}(x_\infty))} \le \rmCh_p^{1/p}(\eta u_\infty) \le \sup_n \|u_n\|_{L^p(\mm_n)}+\|\nabla u_n\|_{L^p(\mm_n)} <\infty.
    \]
    Similarly, the locality of the total variation on open sets yields
    \[
        |Du_\infty|(B_R(x_\infty)) = |D(\eta u_\infty)|(B_R(x_\infty)) \le |D(\eta u_\infty)|(\X_\infty) \le \sup_{n\in\N} \|u_n\|_{L^1(\mm_\infty)} + |Du_n|(\X_n)<\infty.
    \]
    Being $R>0$ arbitrary, we deduce $u_\infty \in W^{1,p}(\X_\infty)$ (resp. $u \in BV(\X_\infty)$) as desired.

    \noindent\textsc{Reduction step}. First, observe that up to considering ${A_R}\coloneqq B_R(x_\infty) \cap A$ and then arguing by monotonicity sending $R\to\infty$, it is enough to prove \eqref{eq:lsc chp on open},\eqref{eq:lsc Du on open} for $A$ bounded open. Moreover, since $u_\infty \in W^{1,p}(\X_\infty)$ (resp. $u_\infty \in BV(\X_\infty)$), by locality we can further assume that $\supp(u_n)$ is equibounded, say contained in $B_R(x_\infty)$, up to replacing $u_n,u_\infty$ respectively with $(\eta u_n),(\eta u_\infty)$ for some $1$-Lipschitz function $\eta$ that is boundedly supported, non-negative and so that $\eta \equiv 1$ on $A$.
    
    It is also possible to reduce to the case in which $ \|u_\infty\|_{L^\infty(\mm_\infty)} \vee \left(\sup_n \|u_n\|_{L^\infty(\mm_n)} \right)<\infty$. Indeed, thanks to the fact that we are supposing equi-bounded supports, we know that actually $u_n$ converges $L^p$-strong to $u_\infty$ (by Proposition \ref{prop:rellich W1p}, and since $L^p$-weak limits are unique). In particular, for all $M>0$ the truncated sequence $u_n^M \coloneqq (-M)\wedge u_n \vee M$ converges as well $L^p$-strong to $u_\infty^M \coloneqq (-M)\wedge u_\infty \vee M$ an \eqref{eq:lsc chp on open},\eqref{eq:lsc Du on open} would follow by monotonicity and the chain rule-argument sending $M\uparrow \infty$.

    All in all, after these reductions steps it sufficient to prove \eqref{eq:lsc chp on open},\eqref{eq:lsc Du on open}  under the additional assumption that $u_n\in L^2(\mm_n),u_\infty \in L^2(\mm_\infty)$ and, by \cite[(e) in Proposition 3.3]{AmbrosioHonda17}, that $u_n$ converges $L^2$-strong to $u_\infty$.

    \noindent\textsc{Proof of \eqref{eq:lsc chp on open}}. Here we assume $p>1$. We shall argue similarly to \cite[Lemma 5.8]{AmbrosioHonda17} and exploit regularization properties of the heat flow on ${\sf RCD}$ spaces. We denote by $h_t^n f$ the heat flow evolution on the space {$\X_n$} starting from $f_n \in L^2(\mm_n)$ at time $t>0$ for every $n \in \bar \N$ (see, e.g. \cite{Gigli14}). Thanks to standard gradient flow estimates on Hilbert spaces and the $L^\infty$-to-Lip regularization in ${\sf RCD}$-setting (see, e.g., \cite[Remark 5.2.11 and Proposition 6.1.6]{GigliPasqualetto20book}), we have
    \[
    \|\nabla h_t^nu_n\|^2_{L^2(\mm_n)} \le \frac{\|u_n\|^2_{L^2(\mm_n)}}{2t},\qquad \|\nabla h^n_tu_n\|_{L^\infty(\mm_n)} \le C_K\frac{\|u_n\|_{L^\infty(\mm_n)}}{\sqrt t},
    \]
    where $C_K>0$ depends only on the uniform Ricci lower bound constant $K\in\R$. In particular, those estimates are uniform in $n\in\N$, recalling also that $\mm_n(\supp \, u_n) \le \mm_n(B_R(x_n))$ for a suitable radius $R>0$ and since $\mm_n(B_R(x_n))$ is converging to some finite value, thanks to the underlying pmGH-convergence. Notice that the latter implies that $h^n_tu_n$ have equi-Lipschitz representatives (by the Sobolev-to-Lipschitz property on ${\sf RCD}$-spaces \cite[Theorem 4.10]{Gigli13_splitting}). By stability properties of the heat flow (cf.\ \cite[Theorem 6.11]{GigliMondinoSavare13}), we know that $h_t^nu_n$ converges $L^2$-strong to $h_t^\infty u_\infty$. By the first estimate in the above, $W^{1,2}$-weak convergence also follows. We are therefore in position to invoke \cite[Lemma 5.8]{AmbrosioHonda17} (that is valid for arbitrary pmGH-converging ${\sf RCD}$-spaces) to deduce that for all $g \in \Lip_{bs}(Z)$ nonnegative, we have
    \begin{equation}
        \int g|\nabla h^\infty_t u_\infty| \,\d\mm_\infty\le \liminf_{n\uparrow\infty} \int g|\nabla h^n_t u_n| \,\d\mm_n,\qquad \forall t>0.
    \label{eq:estim 0}
    \end{equation}
    The above is well defined and finite since $g$ is boundedly supported and we know that $h^n_tu_n$ are equi-Lipschitz. We claim that the above holds also at $t=0$. Indeed, for all $t>0$ we can write 
    \begin{align*}
        \liminf_{n\uparrow\infty}\int g|\nabla u_n|\,\d \mm_n &\ge \liminf_{n\uparrow\infty} \int h^n_t g |\nabla u_n|\,\d \mm_n  - \limsup_{n\uparrow\infty}\int|h^n_tg- g||\nabla u_n|\,\d\mm_n \\
        &\ge   e^{Kt} \liminf_{n\uparrow\infty}\int_U |\nabla h^n_t u_n|\,\d \mm_n  - C\limsup_{n\uparrow\infty}\left(\int|h^n_t g- g|^{p'}\,\d\mm_n\right)^{\frac{1}{p'}}\\
        &\overset{\eqref{eq:estim 0}}{\ge} e^{Kt} \int g |\nabla h^\infty_t u_\infty|\,\d \mm_\infty - C\limsup_{n\uparrow\infty}\left(\int|h^n_t g- g|^{p'}\,\d\mm_n\right)^{\frac{1}{p'}},
    \end{align*}
    having used, in the second line, that the heat flow is {self-}adjoint (see, e.g., \cite[Corollary 5.2.9]{GigliPasqualetto20book}), the $1$-Bakry-\'Emery contraction estimate for a Lipschitz function (c.f. \cite{GigliHan14}) and H\"older inequality with $C\coloneqq \|\nabla u_n\|_{L^p(\mm_n)}$ with $p'$ H\"older conjugate. Now, we notice that $\lim_{t\to 0}\limsup_{n\uparrow\infty}\int|h^n_tg-g|^{p'}\,\d\mm_n =0$ by \cite[Proposition 4.6]{AmbrosioHonda17} using that $h^n_tg$ converges $L^{p'}$-strong to $h^\infty_t g$ in $L^{p'}$-strong by the weak maximum principle and the stability of the heat flow. We can then deduce
    \begin{equation}
     \liminf_{n\uparrow\infty}\int g|\nabla u_n|\,\d \mm_n \ge  \liminf_{t\to 0}e^{Kt} \int g |\nabla h^\infty_t u_\infty|\,\d \mm_\infty \ge  \int g |\nabla u_\infty|\,\d \mm_\infty,\label{eq:lsc chp glsc}
    \end{equation}
    by weak lower semicontinuity, thus proving the claim.
    
    By arbitrariness of $g \in \Lip_{bs}(\Z)$, we directly deduce
    \[
         \liminf_{n\uparrow\infty}\int_U |\nabla u_n|\,\d \mm_n \ge  \int_U |\nabla u_\infty|\,\d \mm_\infty,
    \]
    for every $U\subset \Z$ open and bounded. From this, the claimed estimate \eqref{eq:lsc chp on open} follows now taking into account the following identity
    \[
        \int_A |f|^p\,\d \mm = \sup \sum_k \frac{1}{\mm_\infty(U_k)}\left(\int_{U_k} |f|\,\d \mm_\infty\right)^p,
    \]
    for $f \in L^p(\mm_\infty)$ and where the sup is taken among all partitions $U_k$ of pairwise disjoint bounded open sets of $A$ so that $\mm_\infty(U_k)>0$ (see at the end of the proof of \cite[Lemma 5.8]{AmbrosioHonda17} for $p=2$).   

     \noindent\textsc{Conclusion: proof of \eqref{eq:lsc Du on open}}. Here we consider the case $p=1$ and conclude the proof. Recall that, by the reduction step, we can assume that $u_n \in L^2(\mm_n)$ converges also $L^2$-strong to $u_\infty \in L^2(\mm_\infty)$, that $|u_n|,|u_\infty|\le M$ for some $M>0$ and that $\supp(u_n)\cup \supp(u_\infty)$ are equibounded in $\Z$. Let $t>0$, consider the heat flow evolution $h^n_t u_n$ and recall that $h^n_t  u_n \in \Lip(\X)$ by the $L^\infty$-to-Lip regularization. Again, by the $1$-Bakry-\'Emery contraction for Lipschitz functions (\cite{GigliHan14}), we deduce that for all $t>0$ the sequence $h^n_t u_n$ {is} equi-Lipschitz hence
     \[
     \sup_{n\in\N} \| \nabla h_t^n u_n\|_{L^2(\mm_n)} \le \sup_{n\in\N} \Lip(h^n_tu_n) |D h^n_t u_n|(\X_n) \le   e^{-Kt} \sup_{n\in\N} \Lip(h^n_tu_n)|D u_n|(\X_n)<\infty,
     \]
     where we used the identification result for minimal upper gradients in \cite{GigliHan14}. In particular, we have that $h_t^nu_n$ converges to $h^\infty_t u_\infty$ in $W^{1,2}$-weak, taking also into account the  stability of the heat flow. We can thus combine the estimate
     \[
       \liminf_{n\uparrow\infty} \int_A |\nabla h^n_t u_n|\,\d\mm_n \le e^{-Kt}\liminf_{n\uparrow\infty}|Du_n|(A).
     \]
     with (again by \cite[Lemma 5.8]{AmbrosioHonda17})
     \[
         \liminf_{n\uparrow\infty} \int_A |\nabla h^n u_n|\,\d\mm_n \ge \int_A |\nabla h^\infty_t u_\infty|\,\d\mm_\infty = |Dh_t^\infty u_\infty|(A),
     \]
     to conclude the proof by sending $t\downarrow 0$ and using the lower semicontinuity of the total variation on open sets.
\end{proof}
Notice that, in Proposition \ref{prop:rellich W1p} and in Proposition \ref{prop:lsc on open}, we only used the Gamma-convergence result of \cite[Theorem 8.1]{AmbrosioHonda17}. Hence, by a combination of the two results we can finally upgrade to the Mosco-convergence of Cheeger energies.
\begin{theorem}\label{thm:Mosco Chp}
Let $p\in(1,\infty)$. Then, we have
\begin{itemize}
    \item[${\rm i)}_p$]  if $u_n \in L^p(\mm_n)$ converges $L^p$-weak to some $u_\infty \in L^p(\mm_\infty)$ then
    \[
        \rmCh_p(u_\infty)\le \liminf_{n\uparrow\infty}\rmCh_p(u_n);
    \]
    \item[${\rm ii)}_p$] for every $u_\infty \in L^p(\mm_\infty)$ there is  $u_n \in L^p(\mm_n)$ converging $L^p$-strong to $u_\infty$ and so that
    \[
        \rmCh_p(u_\infty)\ge \limsup_{n\uparrow\infty}\rmCh_p(u_n).
    \]
    \end{itemize}
    Furthermore, we have
    \begin{itemize}
    \item[${\rm i)}_1$] if $u_n \in L^1(\mm_n)$ converges $L^1$-weak to some $u_\infty \in L^p(\mm_\infty)$ then
    \[
        |Du_\infty|(\X_\infty) \le \liminf_{n\uparrow\infty}|Du_n|(\X_n);
    \]
    \item[${\rm ii)}_1$] for every $u_\infty \in L^1(\mm_\infty)$ there is $u_n \in L^1(\mm_n)$ converging $L^1$-strong to $u_\infty$ and so that
    \[
         |Du_\infty|(\X_\infty) \ge \limsup_{n\uparrow\infty}|Du_n|(\X_n);
    \]
    \end{itemize}
\end{theorem}
\begin{proof}
    Conclusions ii)$_p$ and ii)$_1$ are proved in \cite{AmbrosioHonda17}. We shall prove here i)$_p$ and i)$_1$ handling both cases together and assuming that the right hand sides of both conclusions are finite. In this case, up to a not relabeled subsequence, it is not restrictive to assume that eventually $u_n \in W^{1,p}(\X_n)$ (resp.\ $u_n \in BV(\X_n)$) for all $n$ large enough and $\sup_n \rmCh_p(u_n)<\infty$ (resp.\ $\sup_n |Du_n|(\X_n)<\infty$). Finally, we can write \eqref{eq:lsc chp on open},\eqref{eq:lsc Du on open} for an increasing collection of balls $A=B_R(x_\infty)$ and both conclusions follow by monotonicity and taking $R\uparrow \infty$.
\end{proof}
We single out the following technical property of $W^{1,p}$-strong converging sequences for future use.
\begin{lemma}\label{lem:Lpstrong gradient}
    Let $p\in(1,\infty)$ and suppose that $u_n \in W^{1,p}(\X_n)$ converges $W^{1,p}$-weak to some $u_\infty \in W^{1,p}(\X_\infty)$. If $\rmCh_p(u_n)\to\rmCh_p(u_\infty)$, then $|\nabla u_n|$ converges $L^p$-strong to $|\nabla u_\infty|$.
\end{lemma}
\begin{proof}
    Since $\sup_n \| \nabla u_n\|_{L^p(\mm_n)}<\infty$, we infer the existence of a nonnegative function $G \in L^p(\mm_\infty)$ so that $|\nabla u_n|$ converge $L^p$-weak to $G$, along a suitable not relabeled subsequence. Fix now any ball $B\subset \X_\infty$ and consider its $\sfd_\infty$-closure $\overline B$. Clearly, as a subset $\bar B \subset \Z$ it is $\sfd$-closed in $\Z$ as the isometric embedding is a closed map. Since $|\nabla u_n| \mm_n$ converges weakly to $G\mm_\infty$ in duality with $C_{bs}(\Z)$, and since boundaries of balls are negligible by Bishop-Gromov, by weak upper semicontinuity on closed sets we can write    
    \[
    \int_B G^p\,\d \mm_\infty = \int_{\overline{B}} G^p \,\d \mm_\infty \ge \limsup_{n\uparrow\infty}  \int_{\overline{B}} |\nabla u_n|^p\,\d \mm_n \ge  \liminf_{n\uparrow\infty}  \int_B |\nabla u_n|^p\,\d \mm_n \overset{\eqref{eq:lsc chp on open} }{\ge} \int_B|\nabla u_\infty|^p\,\d\mm_\infty.
    \]
    By arbitrariness of $B$, we therefore deduce that $|\nabla u_\infty|\le G$ at $\mm_\infty$-a.e.\ point. However, by $L^p$-weak lower semicontinuity and the current assumptions, we get
    \[
    \| G\|_{L^p(\mm_\infty)}\le \liminf_{n\uparrow\infty}\|\nabla u_n\|_{L^p(\mm_n)} = \rmCh_p^{1/p}(u_\infty) \le \|G\|_{L^p(\mm_\infty)}.
    \]
    Therefore, all the inequalities are equalities, giving in turn that $G=|\nabla u_\infty|$ $\mm_\infty$-a.e.\ and that $|\nabla u_n|$ converges $L^p$-strong to $|\nabla u_\infty|$. Moreover, being the limit independent of the subsequence chosen at the beginning, this occurs along the original sequence. The proof is therefore concluded.
\end{proof}
We conclude with the analogue property for the BV case.
\begin{lemma}\label{lem:strict covnergence}
    Suppose that $u_n \in BV(\X_n)$ converges $L^1$-weak to some $u_\infty \in BV(\X_\infty)$ and that $|Du_n|(\X_n) \to |D u_\infty|(\X_\infty)$. Then $|D u_n| \weakto |D u_\infty|$ in duality with $C_b(\Z)$.
\end{lemma}
\begin{proof}
    This follows by standard characterization of weak convergence of finite nonnegative measures using, in this setting, the lower semicontinuity on open sets \eqref{eq:lsc Du on open} and Cavalieri's formula (see, e.g., the arguments in the proof of \cite[Proposition 4.5.6]{DiMarinoPhD}).
\end{proof}
\subsection{Technical results for locally Sobolev functions}
We extend some technical convergence results to the case of locally Sobolev functions. This is necessary for the goal of this note, as a Sobolev inequality of Euclidean type implies global integrability for a different exponent from that of the gradient.

We shall need the following lower semicontinuity result of gradient norms of locally Sobolev functions, using Theorem \ref{thm:Mosco Chp} that is now available.
\begin{proposition}\label{prop:Ch is Lp-lsc}
Let $p \in (1,\infty)$ and suppose $u_n\in W^{1,p}_{loc}(\X_n)$ converges $L^p_{loc}$-strong to $u_\infty$. Then 
\begin{equation}
      \|\nabla u_\infty \|^p_{L^p(\mm_\infty)} \le \liminf_{n\uparrow\infty} \| \nabla u_n \|^p_{L^p(\mm_n)},
\label{eq:Ch is Lp-lsc}  
\end{equation}
meaning that, if the right hand side is finite, then $u_\infty\in W^{1,p}_{loc}(\X_\infty)$ and \eqref{eq:Ch is Lp-lsc} holds. 
\end{proposition}
\begin{proof}
If the right hand side in \eqref{eq:Ch is Lp-lsc} is infinite then there is nothing to prove, so let us assume it to be finite. Fix any ball $B\subset Z$  and take $\eta \in \Lip_{bs}(\Z)$ constantly equal to $1$ on $B$. Since $\eta u_n$ converges $L^p$-strong to $\eta u_\infty$, Proposition \ref{prop:lsc on open} yields
\[
    \int_B  |\nabla u_\infty |^p \, \d \mm_\infty = \int_B |\nabla (\eta u_\infty) |^p \, \d \mm_\infty  \le \liminf_n \int_B  |\nabla (\eta u_n)|^p \, \d \mm_n \le \liminf_n\|\nabla u_n\|^p_{L^p(\mm_n)}<\infty,
\]
where in the first and last step we used the locality of weak upper gradients. By the arbitrariness of $B$, the proof follows.
\end{proof}
A direct corollary of the compactness results in Proposition \ref{prop:rellich W1p} and the above lower semicontinuity property is the following local compactness that we single out for later use.
\begin{lemma}\label{lem:pmGHW1ploc}
Let $p,q\in(1,\infty)$ with $q\ge p$ and suppose $u_n \in W^{1,p}_{loc}(\X_n)$ converges $L^q$-weak to $u_\infty \in L^q(\mm_\infty)$ and $\sup_n \|\nabla u_n\|_{L^p(\mm_n)} <\infty$. Then, up to  a subsequence $u_n$ converges $L^p_{loc}$-strong to $u_\infty \in W^{1,p}_{loc}(\X_\infty)$ with $|\nabla u_\infty|\in L^p(\mm_\infty)$. Finally, if also $\|\nabla u_n\|_{L^p(\mm_n)} \to \|\nabla u_\infty\|_{L^p(\mm_\infty)}$, then also $|\nabla u_n|$ converges $L^p$-strong to $|\nabla u_\infty|$.
\end{lemma}
\begin{proof}
We first prove the $L^p_{loc}$-strong convergence. Consider $\eta \in \Lip_{bs}(\Z)$ (recall that $(\Z,\sfd)$ is a space realizing the convergence). Notice that the sequence $\eta u_n$ satisfies $\supp(\eta u_n)\subset B_R(x_n)$ for some fixed $R>0$ independent on $n\in\N$. Since $q\ge p$, by H\"older inequality and the Leibniz rule we have  $\sup_n \| \eta u_n\|_{W^{1,p}(\mm_n)} <+\infty$. Thus by Proposition \ref{prop:rellich W1p}, there exists a subsequence $(n_k)$ such that  $\eta u_{n_k}$ converges $L^p$-strong to some $v \in W^{1,p}(\X_\infty)$, which must be equal to $\eta u_\infty$ by uniqueness of weak limits. In particular, Proposition \ref{prop:Ch is Lp-lsc} guarantees that $u_\infty \in W^{1,p}_{loc}(\X_\infty)$ with $|\nabla u_\infty|\in L^p(\mm_\infty)$. This shows the first part of the statement. For the second part we assume that $\|\nabla u_n\|_{L^p(\mm_n)} \to \|\nabla u_\infty\|_{L^p(\mm_\infty)}$.  By considering any ball $B\subset\X_\infty\subset \Z$ and $\eta \in \Lip_{bs}(\Z)$ with $\eta\equiv 1$ on $B$, we can argue as in  the proof of Lemma \ref{lem:Lpstrong gradient}:
\begin{align*}
     \int_B G\,\d \mm_\infty &= \int_{\overline{B}} G \,\d \mm_\infty \ge \limsup_{n\uparrow\infty}  \int_{\overline{B}} |\nabla u_n|^p\,\d \mm_n \ge \limsup_{n\uparrow\infty}  \int_B |\nabla (\eta u_n)|^p\,\d \mm_n \\
     &\overset{\eqref{eq:lsc chp on open}}{\ge} \int_B|\nabla (\eta u_\infty)|^p\,\d\mm_\infty = \int_B|\nabla  u_\infty|^p\,\d\mm_\infty,
\end{align*}
where $G$ is any $L^p$-weak limit of $|\nabla u_n|$, which exists up to further passing to a subsequence. Notice that, in the application of \eqref{eq:lsc chp on open}, we are using that $\eta u_n$ converges $L^p$-weak to $\eta u_\infty$ (actually, also $L^p$-strong, under the current assumptions) and $\sup_n \|\eta u_n\|_{W^{1,p}(\X_n)}<\infty$ by the Leibniz rule. This concludes the proof, by arbitrariness of $B$, by the same reasoning as at the end of Lemma \ref{lem:Lpstrong gradient}.
\end{proof}
Next, we show the existence of certain recovery sequences.
\begin{lemma}\label{lem:strongrecov}
Let $p,q \in (1,\infty)$ with $q\ge p$ and  $u_\infty \in W^{1,p}_{loc}(\X_\infty)\cap L^{q}(\mm_\infty)$ with  $|\nabla u_\infty|\in L^p(\mm_\infty)$. Then, there exists $u_n \in W^{1,p}_{loc}(\X_n)\cap L^{q}(\mm_n)$ that converges $L^{q}$-strong and $L^p_{loc}$-strong to $u_\infty$ and so that $|\nabla u_n|$ converges $L^p$-strong to $|\nabla u_\infty|$.
\end{lemma}
\begin{proof}
By \cite[Lemma 3.2]{NobiliViolo24} (holding also for $p\neq 2$) we can find a sequence $u_n \in W^{1,p}(\X_\infty)\cap L^q(\mm_\infty)$ such that $u_n \to u_\infty$ and $|\nabla u_n|\to |\nabla u_\infty|$ strongly in $L^p(\mm_\infty).$ From \cite[Lemma 6.4]{NobiliViolo21}  (there written for compact spaces and for $p=2$, but the same proof works in the present setting)  there exists a sequence $u_n^k \in W^{1,p}(\X_n)$ that converges $L^q$-strong and $W^{1,p}$-strong to $u_n$ as $k\uparrow \infty$. Then, the sought $L^q$-strong convergence follows by a diagonal argument, while the $L^p_{loc}$-strong convergence follows from Lemma \ref{lem:pmGHW1ploc}. Finally, the $L^p$-strong convergence of $|\nabla u_n|$ follows by the last conclusion in Lemma \ref{lem:pmGHW1ploc}.
\end{proof}
We conclude this part by showing that there is a linear convergence of gradients of locally Sobolev functions. 
\begin{proposition}[Linearity]\label{prop:linearity W1ploc convergence}
    Let $p \in (1,\infty)$ and suppose that $u_n,v_n \in W^{1,p}_{loc}(\X_n)$ both converges  $L^p_{loc}$-strong to $u_\infty \in W^{1,p}_{loc}(\X_\infty)$. If $\|\nabla u_n\|_{L^p(\mm_n)}\to \|\nabla u_\infty\|_{L^p(\mm_\infty)}$ and $\|\nabla v_n\|_{L^p(\mm_n)}\to \|\nabla u_\infty\|_{L^p(\mm_\infty)}$ as $n\uparrow\infty$, then we have
    \[
        \lim_{n\uparrow\infty}\|\nabla (u_n -v_n)\|_{L^p(\mm_n)} = 0.
    \]
\end{proposition}
\begin{proof}
The statement is known if $p=2$ when $u_n,v_n$ converges $L^2$-strong, i.e.\ the $W^{1,2}$-strong convergence is linear. This simply follows by cosine law for the $2$-weak upper gradients (having assumed {infinitesimally} Hilbertianity) and the convergence of the couplings \cite[Eq. 5.3]{AmbrosioHonda17}. We handle here the arbitrary exponent case.

First, since $\X_n$ are assumed ${\sf RCD}$ spaces, we can appeal to the Clarkson inequalities (see \cite[Eq. (4.3)]{GigliNobili21} to write for all $n\in\N$: if $p\ge 2$ then
\[
    \left\|\nabla \left(\frac{u_n -v_n}{2}\right)\right\|^p_{L^p(\mm_n)}+  \left\|\nabla \left(\frac{u_n +v_n}{2}\right)\right \|^p_{L^p(\mm_n)} \le \frac{1}{2}\|\nabla u_n\|_{L^p(\mm_n)}^p +  \frac{1}{2}\|\nabla v_n\|_{L^p(\mm_n)}^p,
\]
while, if $p\in (1,2)$, denoting by $q$ the H\"older conjugate, we have
\[
    \left\|\nabla \left(\frac{u_n -v_n}{2}\right)\right\|^q_{L^p(\mm_n)}+  \left\|\nabla \left(\frac{u_n +v_n}{2}\right)\right \|^q_{L^p(\mm_n)} \le \left(\frac{1}{2}\|\nabla u_n\|^p_{L^p(\mm_n)} +  \frac{1}{2}\|\nabla v_n\|^p_{L^p(\mm_n)}\right)^{\frac qp}.
\]
For the validity of the above, the relevant fact is that $\X_n$ are infinitesimal Hilbertian spaces and that weak upper gradients do not depend on the integrability exponent in a weak sense (see \cite{GigliNobili21}). 
By these  inequalities, in the whole range $p \in (1,\infty)$, the conclusion of the proof will be achieved provided we can show that
\[
    \lim_{n\uparrow\infty}\|\nabla (u_n +v_n)\|_{L^p(\mm_n)}= 2\|\nabla u_\infty\|_{L^p(\mm_\infty)}.
\]
The above will directly follow from the chain of inequalities
\begin{align*}
    2\|\nabla u_\infty\|_{L^p(\mm_\infty)} &\overset{(\ast)}{\le}\liminf_{n\uparrow\infty}\|\nabla (u_n +v_n)\|_{L^p(\mm_n)} \\
    &\le \lim_{n\uparrow\infty}\|\nabla u_n\|_{L^p(\mm_n)}  + \lim_{n\uparrow\infty}\|\nabla v_n\|_{L^p(\mm_n)} =2\|\nabla u_\infty\|_{L^p(\mm_\infty)},
\end{align*}
provided $(\ast)$ is true. However, $(\ast)$ follows by Proposition \ref{prop:Ch is Lp-lsc} and noticing that the $L^p_{loc}$-strong convergence is linear (simply notice that $\eta(u_n+v_n)$ converges to $2\eta u_\infty$ for every $\eta \in \Lip_{bs}(\Z)$, whence $u_n+v_n$ converges $L^p_{loc}$-strong to $2u_\infty$).
\end{proof}
\section{Concentration compactness principles}
In this part, we extend for an exponent $p\neq 2$ the concentration compactness principles studied in \cite{NobiliViolo21,NobiliViolo24}. We state the main result and provide the proof at the end of this section.
\begin{theorem}\label{thm:CC extremal}
    For every $N\in (1,\infty)$ and $p \in (1,N)$, there   exists $\eta_{p,N}\in(0,1/2)$ such that the following holds. Let $(Y_n,\rho_n,\mu_n,y_n)$ be pointed ${\sf RCD}(0,N)$ spaces. Set $p^*=pN/(N-p)$. Suppose that for some $A_n\to A \in(0,\infty)$ it holds
   \begin{equation}
   \|u\|_{L^{p^*}(Y_n)} \le A_n\|\nabla u\|_{L^p(Y_n)},\qquad \forall u \in W^{1,p}(Y_n).\label{eq:convention}
   \end{equation}
Furthermore, suppose there are non-constant functions $u_n \in W^{1,p}(Y_n)$ with $\| u_n\|_{L^{p^*}(\mu_n)} =1$  and
\begin{align}
&\sup_{y \in Y_n} \int_{B_1(y)}|u_n|^{p^*}\,\d\mu_n=\int_{B_1(y_n)}|u_n|^{p^*}\,\d\mu_n = 1-\eta,\label{eq:Levyscalings} \\
 &\| u_n\|_{L^{p^*}(\mu_n)} \ge  \tilde A_n \|\nabla u_n\|_{L^p(\mu_n)},\label{eq:extremals}
\end{align}
for  some $\tilde A_n \to A$  and  $\eta \in(0,\eta_{p,N}).$ Then, up to a subsequence, we have:
    \begin{itemize}
    \item[ \rm i)]
    there is a pointed $\RCD(0,N)$-space $(Y,\rho,\mu,y)$ so that 
    \[
    Y_n\overset{pmGH}{\to} Y,
    \]
    and it holds
    \begin{equation}
       \|u\|_{L^{p^*}(Y)} \le A\|\nabla u\|_{L^p(Y)},\qquad \forall u \in W^{1,p}(Y);\label{eq:sob limit Y}
       \end{equation}
    \item[ \rm ii)]  $u_n$ converges  $L^{p^*}$-strong to some $0\neq u_\infty \in W^{1,p}_{loc}(Y)$ with $|\nabla u_\infty| \in L^p(\mu)$ and
    \[
       \int |\nabla u_n|^p\, \d \mu_n \to  \int |\nabla u_\infty|^p\,\d\mu, \qquad  \text{as }n\uparrow\infty;
    \]
    \item[ \rm iii)] it holds
    \[
        \| u_\infty\|_{L^{p^*}(\mu)} =  A \|\nabla u_\infty\|_{L^p(\mu)}.
    \]
\end{itemize}
\end{theorem}
Compare the above to \cite[Theorem 6.2]{NobiliViolo24} and notice that here we drop the $B$-term in the Sobolev inequality, as this is not needed in this note. This will slightly simplify some arguments. 

\subsection{Decomposition principle}
We study here a decomposition principle describing concentration phenomena of sequences of functions and measures arising from Sobolev inequalities. This extends \cite[Lemma A.7]{NobiliViolo24} (in turn relying on \cite[Lemma 6.6]{NobiliViolo21}) for $p\neq 2$, but in the absence of the $B$-term in the Sobolev inequality that we shall never need this in this note. The proof is similar, but we include all the details to handle the general exponent and to explicitly highlight where the technical machinery of Section \ref{sec:Mosco} will be needed.
\begin{lemma}\label{lem:conccomp 2}
Let $(\X_n,\sfd_n,\mm_n,x_n)$ be  pointed $\RCD(K,N)$ spaces, $n \in\N\cup\{\infty\}$, for some $K \in \R$, $N \in(1,\infty)$  with $\X_n\overset{pmGH}{\to}\X_\infty$ and assume that \eqref{eq:convention} holds for some $A_n>0$ uniformly bounded and $p\in (1,N)$.

Suppose further that $u_n \in W_{loc}^{1,p}(\X_n) \cap  L^{p^*}(\mm_n)$ with $\sup_n \|\nabla u_n\|_{L^p(\mm_n)}<\infty $ is $L^p_{loc}$-strong converging to $u_\infty \in L^{p^*}(\mm_\infty)$ and suppose that $|\nabla u_n|^p \mm_n \weakto \omega,$ $|u_n|^{p^*}\mm_n\weakto \nu$ in duality with $C_{bs}(\Z)$ and $C_b(\Z)$, respectively (where $(\Z,\sfd)$ is a fixed realization of the convergence). 
    
Then, $u_\infty \in W^{1,p}_{loc}(\X_\infty)$ with $|\nabla u_\infty|\in L^p(\mm_\infty)$ and
\begin{itemize}
    \item[{\rm i)}] there are a countable set $J$, points $(x_j)_{j\in J}\subset \X_\infty$ and  $(\nu_j)_{j\in J}\subset\R^+$ so that
    \[ \nu =|u_\infty|^{p^*}\mm_\infty +\sum_{j\in J}\nu_j\delta_{x_j};\]
    \item[{\rm ii)}] there is $(\omega_j)_{j\in J} \subset \R^+$ satisfying $\nu_j^{p/p^*} \le (\limsup_{n}A_n)\omega_j$ for every $j\in J$ and such that
    \[\omega \ge |\nabla u_\infty|^p\mm_\infty +\sum_{j\in J}\omega_j\delta_{x_j}.\]
    In particular, we have $\sum_j \nu_j^{p/p^*}<\infty$.
    \end{itemize}
\end{lemma}
\begin{proof}
By assumptions, we can also assume that $u_n$ is $L^{p^*}$-weak converging to $u_\infty$ (by uniqueness of limits), simply by plugging $u_n$ in \eqref{eq:convention} to deduce a uniform $L^{p^*}$-bound. We subdivide the proof into two steps.

\noindent{\sc Step 1}. Suppose first that $u_\infty=0$. Let $\varphi \in \LIP_{bs}(\Z)$ and plugging $\varphi u_n\in W^{1,p}(\X_n)$ in \eqref{eq:convention} yields
\[ \left( \int |\varphi|^{p^*}|u_n|^{p^*}\, \d \mm_n \right)^{\frac{1}{p^*}} \le A_n\left(\int |\nabla(\varphi u_n)|^p\, \d \mm_n\right)^{\frac 1p}, \qquad \forall n\in \N.\]
Thanks to the weak convergence, the left hand side of the inequality tends to $(\int |\varphi|^{p^*}\,\d \nu )^{1/p^*}$. Instead, estimating by the Leibniz-rule  $\int |\nabla(\varphi u_n)|\, \d \mm_n \le \int |\nabla \varphi||u_n| + |\phi||\nabla u_n|\, \d \mm_n$ and using the $L^p_{loc}$-strong convergence to deduce $\int |\nabla \varphi|^p|u_n|^p\d \mm_n\to0$, we get
\[
\left( \int |\varphi|^{p^*} \d \nu \right)^{1/p^*} \le \limsup_n A_n \left(\int |\varphi|^p\, \d \mu\right)^{1/p}, \qquad \forall \varphi \in \LIP_{bs}(\Z).
\]
The application of \cite[Lemma 6.5]{NobiliViolo21} in $\Z$ gives conclusions i),ii) for the case $u_\infty=0$. Notice that the points $(x_j)_{j\in J}$ (which are a-priori in $\Z$) can be proved to belong actually to $\X_\infty$ noticing, for any $j \in J$ that the  weak convergence $|u_n|^{p^*}\mm_n \rightharpoonup\nu$ implies the existence of a sequence $y_n \in \supp(\mm_n)=\X_n$ such that $\sfd_\Z(y_n,x_j)\to0.$ Then the pmGH-convergence of $\X_n$ to $\X_\infty$ ensures that $x_j \in \X_\infty$.

\noindent \textsc{Step 2}. Here $u_\infty$ is possibly nonzero. By stability of Sobolev inequalities (c.f.\ \cite[Lemma 4.1]{NobiliViolo21}), we know that \eqref{eq:convention} holds in $\X_\infty$ with $A\coloneqq \limsup_n A_n$. From Lemma \ref{lem:strongrecov} there exists a sequence $\tilde u_n \in W^{1,p}_{loc}(\X_n)$ such that $\tilde u_n$ converges in $L^{p^*}$-strong and $L^{p}_{loc}$-strong to $u_\infty$ and $|\nabla u_n|$ converges $L^p$-strong to $|\nabla u_\infty|$. For every $\varphi \in \Lip_{bs}(\Z)$ nonnegative, by the Brezis-Lieb lemma \cite[Lemma A.1]{NobiliViolo24} we deduce
\begin{equation}
\lim_{n\uparrow\infty}\int |\varphi|^{p^*}|u_n|^{p^*}\,\d\mm_n - \int |\varphi|^{p^*}|u_n-\tilde u_n|^{p^*}\,\d\mm_n = \int |\varphi|^{p^*}|u_\infty|^{p^*}\,\d\mm_\infty.
\label{eq:step BL}
\end{equation}
Consider now the sequence $v_n := u_n-\tilde u_n$. Clearly $v_n$ converge $L^{p}_{loc}$-strong and $L^{p^*}$-weak to zero. Moreover, by tightness using the estimates $ |v_n|^{p^*} \le 2^{p^*}(|u_n|^{p^*}+|\tilde u_n|^{p^*})$ and $|\nabla v_n|^p\le 2^p(|\nabla u_n|^p+|\nabla \tilde u_n|^p) $ we deduce, up to a subsequence, that {$| v_n|^{p^*}\mm_n$} weakly converges in duality with $C_b(\Z)$ to some probability measure $\tilde \nu$ and $|\nabla v_n|^p\mm_n$ weakly converges {in duality with $C_{bs}(\Z)$} to some finite Borel measure $\tilde w$. Thus, Step 1 applies for the sequence $v_n$ and the conclusions i),ii) hold for the measures $\tilde \nu,\tilde \omega$, for suitable countable set $J$, points $(x_j)\subset \X_\infty$ and weights $(w_j)\subset \R^+$. Then, conclusion i) for $\nu $ is immediate recalling \eqref{eq:step BL} with the underlying weak convergence. 

We are left to show ii) for $\omega$. We claim that
\begin{align*}
    &\omega(\{x_j\}) = \tilde \omega(\{ x_j\})\ge \omega_j,\qquad \forall j \in J,\\
    &\omega \ge |\nabla u_\infty|\mm_\infty.
\end{align*}
Clearly, by mutual singularity the combination of the two would conclude the proof. Fix $j \in J$ and $\eps >0$, consider $\nchi_\eps \in \LIP_{bs}(\Z)$, $0\le \nchi_\eps \le 1, \nchi_\eps(x_j)=1$ and supported in $B_\eps(x_j)$ and let us estimate
\begin{equation*}
    \begin{split}
       \Big| \int \nchi_\eps|\nabla u_n|^p\, \d \mm_n -\int \nchi_\eps|\nabla v_n|^p\,\d\mm_n\Big| &\le  p \int\nchi_\eps  \big||\nabla u_n|-|\nabla v_n|\big|\big(|\nabla u_n|^{p-1}+|\nabla v_n|^{p-1}\big)\, \d \mm_n \\
        &\le p \int \nchi_\eps|\nabla \tilde u_n|\big(|\nabla u_n|^{p-1}+|\nabla v_n|^{p-1}\big)\, \d \mm_n \\
       &\le p \left(\int\nchi_\eps^p|\nabla \tilde u_n|^p\, \d \mm_n\right)^{1/p}\left(\| \nabla u_n\|_{L^p(\mm_n)}^{\frac{p-1}{p}} +\| \nabla v_n\|_{L^p(\mm_n)}^{\frac{p-1}{p}}\right).
    \end{split}
\end{equation*}
Recalling that $|\nabla \tilde u_n|\to |\nabla u_\infty|$ $L^p$-strong by Lemma \ref{lem:strongrecov},  we deduce that $\int\nchi_\eps^p|\nabla\tilde u_n|^p\, \d \mm_n\to \int\nchi_\eps^p|\nabla u_\infty|^p\, \d \mm_\infty$. Moreover $\int\nchi_\eps^p|\nabla u_\infty|^p\, \d \mm_\infty\to0$ as $\eps\to 0^+$ and $|\nabla u_n|,|\nabla v_n|$ are uniformly bounded in $L^p(\mm_n)$. Therefore taking first $n \to +\infty$ and afterwards $\eps\to 0^+$ we get $\omega(\{x_j\}) = \tilde \omega(\{ x_j\})$. Now, since $\omega$ is non-negative and $\tilde \omega \ge \sum_{j \in J}\omega_j\delta_{x_j}$ the first claim follows.

We next prove the second claim. We fix $\phi \in \Lip_{bs}(\Z)$, $\phi \ge 0$, and $\nchi \in \LIP_{bs}(\Z)$ be such that $\nchi=1$ in $\supp (\phi)$. It is easy to check that $\nchi u_n$ is $W^{1,p}$-weak converging to $\nchi u_\infty$ (recall that $u_n\to u_\infty$ in $L^p_{loc}$). Then, by Proposition \ref{prop:lsc on open} (see \eqref{eq:lsc chp glsc}) and locality we get 
\[
     \int \phi |\nabla u_\infty|^p\, \d \mm_\infty=\int \phi |\nabla (\nchi u_\infty)|^p\, \d \mm_\infty \le \liminf_{n\uparrow\infty}\int \phi |\nabla (\nchi u_n)|^p\, \d \mm_n= \int \phi \,\d\omega.
\]
By arbitrariness of $\varphi$, the  second claim is proved and, as discussed, the proof is  concluded.
\end{proof} 

\subsection{Proof of concentration compactness} Here we finally combine the previous technical results and prove the main concentration compactness principle.
\begin{proof}[Proof of Theorem \ref{thm:CC extremal}]  We subdivide the proof into different steps.

\noindent{{\sc Step 1}}. We take $\eta_{N,p}\coloneqq \frac{\lambda_{0,N,p}}8\wedge  \frac13$, with $\lambda_{0,N,p}$ as in Lemma \ref{lem:density bound}. To show i), we first need to check the assumptions of the Gromov pre-compactness result in this setting (c.f.\ Theorem \ref{thm:gromov}), i.e.\ we check that $\mu_n(B_1(y_n)) \in (v^{-1},v) $ for some $v>1.$  If we can prove that $\diam(Y_n)> \eta_{N,p}^{-1}\ge 8\lambda_{0,N,p}^{-1}$, then the assumptions \eqref{eq:Levyscalings} and \eqref{eq:extremals} make it possible to invoke Lemma \ref{lem:density bound} (for $C_{0,N,p}>0$ as in this Lemma) to obtain
\[
\limsup_n \mu_n(B_1(y_n))\le \limsup_n \frac{C_{0,N,p}}{ (\tilde A_n)^{N/p}}= \frac{C_{0,N,p} }{ A^{N/p}}<+\infty.
\]
However, this is actually trivially true since the validity of \eqref{eq:convention} implies that $\diam(Y_n)=+\infty$, see \cite[Theorem 4.6]{NobiliViolo21}. On the other hand, plugging in \eqref{eq:convention} the functions $\phi_n\in \LIP(Y_n)$ such that $\phi_n=1$ in $B_1(y_n)$ with $\supp \varphi_n \subset B_2(y_n)$, $0\le \phi_n\le 1$ and $\Lip(\phi_n)\le 1$, we get
\[
\mu_n(B_1(y_n))^{p/p^*}\le A^p\mu_n(B_2(y_n))\le 2^NA^p\mu_n(B_1(y_n)),
\]
where we used the doubling property of $\mu_n$ (see \cite[Corollary 2.4]{Sturm06II}). Since $A_n \to A>0$ we also obtain $\liminf_n \mu_n(B_1(y_n))>0.$ Therefore up to a not relabeled subsequence, we deduce that $Y_n$ pmGH converge to some pointed $\RCD(0,N)$ space $(Y,\rho,\mu,y)$. Moreover, the stability of the Sobolev inequalities \cite[Lemma 4.1]{NobiliViolo21} guarantees that \eqref{eq:sob limit Y} holds. This settles point i).

\noindent{{\sc Step 2}}. From now on we assume to have fixed a realization of the convergence in a proper metric space $(\Z,\sfd)$. Set $\nu_n \coloneqq  |u_n|^{p^*}\mu_n$. Up to a subsequence, (exactly) one of cases i),ii),iii) in \cite[Lemma A.6]{NobiliViolo24} holds. We claim i) (i.e.\ compactness) occurs. First, notice that vanishing as in case ii) cannot occur:
\[ \limsup_{n\uparrow\infty} \sup_{y \in Y_n }\nu_n(B_R(y)) \ge \limsup_{n\uparrow\infty} \nu_n(B_1(y_n)) \stackrel{\eqref{eq:Levyscalings}}=1-\eta,\qquad  \forall R\ge 1.
\]
Thus, it remains to exclude the dichotomy case iii). Suppose by contradiction that  this occurs for some $\lambda \in (0,1)$ (with $\lambda \ge \limsup_{n}\sup_z \nu_n(B_R(z))$ for all $R>0$), sequences $R_n\uparrow \infty$, $(z_n)\subset \Z$ and measures $\nu_n^1,\nu_n^2$ so that
\[
    \begin{split}
        &0 \le \nu_n^1+\nu_n^2 \le \nu_n, \\
        &\supp( \nu_n^1) \subset B_{R_n}(z_n), \quad \supp( \nu_n^2) \subset \Z \setminus B_{10R_n}(z_n), \\
        & \limsup_{n\uparrow\infty}   \ \big| \lambda -  \nu_n^1( \Z) \big|+ \big| (1-\lambda ) -  \nu_n^2( \Z) \big|  =0 .
    \end{split}
\]
We claim first that $\supp (\nu_n^1) \subset B_{3R_n}(y_n)$ and $\supp (\nu_n^2) \subset \Z\setminus B_{4R_n}(y_n)$. Indeed $\lambda \ge \limsup_{n\uparrow\infty} \nu_n(B_1(y_n)) = 1-\eta$ and 
$$\liminf_{n\uparrow\infty} \nu_n(B_{R_n}(z_n)) \ge\liminf_{n\uparrow\infty}  \nu_n^1(B_{R_n}(z_n))=\lim_{n\uparrow\infty} \nu_n^1(\Z) = \lambda\ge 1-\eta.$$  Since $\nu_n(B_1(y_n))=1-\eta$ and $\eta<1/2$, this implies that for $n$ large enough $B_{R_n}(z_n) \cap B_1(y_n) \neq 0$, which implies the claim, provided $R_n\ge 1.$

Let $\varphi_n$ be a Lipschitz cut-off so that $0\le  \varphi_n \le 1, \varphi_n \equiv 1$ on $B_{3R_n}(y_n)$, $\supp (\varphi_n)\subset B_{4R_n}(y_n)$  and $\Lip(\varphi_n)\le R_n^{-1}$, for every $n \in \N$. Since 
\begin{equation}
1\ge |\varphi_n|^p + |(1-\varphi_n)|^p,\qquad\text{in $\Z$},\label{eq:phi one minus phi}
\end{equation} 
we can estimate by triangular inequality, the Leibniz rule and  Young inequality
\begin{equation}
\begin{split}
      \|\nabla u_n\|^p_{L^p(\mu_n)}  &\ge \| \varphi_n|\nabla u_n| \|^p_{L^p(\mu_n)} +  \| (1-\varphi_n) |\nabla u_n|  \|^p_{L^p(\mu_n)} \\
       &\ge \| \nabla (u_n\varphi_n)\|_{L^p(\mu_n)}^p+ \|\nabla (u_n(1-\varphi_n)) \|_{L^p(\mu_n)}^p - R_n(\delta)
\end{split}
  \label{eq:estim 1}
\end{equation}
for every $\delta>0$ and every $n$, where the reminder $R_n(\delta)$ can be estimated, for a suitable constant $C_p>0$, as follows
\[
R_n(\delta) \le \frac{ C_p +1}{\delta^p}\|u_n|\nabla \varphi_n|\|^p_{L^p(\mm_n)} + C_p \delta^{\frac{p}{p-1}} \|\nabla u_n\|^p_{L^p(\mm_n)}.
\]
Setting $O_n\coloneqq  B_{4R_n}(y_n)\setminus B_{3R_n}(y_n)$, we have by the H\"older inequality
\[ \| u_n |\nabla \varphi_n| \|^p_{L^p(\mu_n)} \le {R_n}^{-p}\| u_n\|^p_{L^{p^*}(O_n)} \mu_n(O_n)^{p/N} \le   4^p v^{p/N}\| u_n\|^p_{L^{p^*}(O_n)},   \]
having used that $\mu_n(O_n) \le \mu_n(B_{4R_n}(y_n))\le (4R_n)^N\mu_n(B_1(y_n))\le    (4R_n)^Nv$, by the Bishop-Gromov inequality, for suitable $v>0$ as found in Step 1. Notice that we also have
\[ \limsup_{n\uparrow\infty} \| u_n\|_{L^{p^*}(O_n)} \le  \limsup_{n\uparrow\infty}\Big| 1 -  \nu^1_n(\Z)  - \nu^2_n(\Z)  \Big|^{1/p^*}  =0,\]
from which we get $ \lim_n \| u_n |\nabla \varphi_n| \|^p_{L^p(\mu_n)} =0$. Therefore, recalling that $\|\nabla u_n\|^p_{L^p(\mu_n)}$ is uniformly bounded by \eqref{eq:extremals}, choosing appropriately $\delta_n\to 0$,  we get
\begin{equation}\label{eq:zero reminder}
    R_n(\delta_n)\to 0.
\end{equation}
Thus, recalling that $\lim_n A_n = \lim_n \tilde A_n$, we get
\begin{align*}
       1 = \|u_n\|_{L^{p^*}(\mm_n)}^p &\overset{\eqref{eq:extremals},\eqref{eq:estim 1},\eqref{eq:zero reminder}}{\ge}  \limsup_{n\uparrow\infty}   A^p_n  \| \nabla (u_n\varphi_n) \|^p_{L^p(\mu_n)}   +  A^p_n\| \nabla (u_n(1-\varphi_n)) \|^p_{L^p(\mu_n)}    \\
      &\overset{\eqref{eq:convention}}{\ge}  \limsup_{n\uparrow\infty}    \| u_n\varphi_n\|^p_{L^{p^*}(\mu_n)} +  \| u_n(1-\varphi_n)\|^p_{L^{p^*}(\mu_n)}  \\
       &\ge \limsup_{n\uparrow\infty} \big( \nu^1_n(\Z)\big)^{p/p^*} +\big( \nu^2_n(\Z)\big)^{p/p^*} \\
      &\ge \lambda ^{p/p^*} +(1-\lambda)^{p/p^*} >1,
\end{align*}
having used the strict concavity of $t \mapsto t^{p/p^*}$ and the fact that $\lambda \in (0,1)$. This gives a contradiction, hence the dichotomy case cannot occur.

\noindent{\sc Step 3}. In the previous step, we proved thus that i) in \cite[Lemma A.6]{NobiliViolo24} occurs, i.e.\  there {is a} $(z_n) \subset \Z$ so that for every $\eps>0$ there exists $R\coloneqq R(\eps)$ so that $\int_{B_R(z_n)} |u_n|^{p^*}\,\d\mu_n \ge 1-\eps$ for all $ n \in \N$. If $\eps <1/2$, we have $B_R(z_n) \cap B_1(y_n) \neq \emptyset$ and
\begin{equation}
    \int_{B_{2R+1}(y_n)} |u_n|^{p^*}\,\d\mu_n \ge 1-\eps\qquad \forall n \in \N.\label{eq:usigmaCompactness}
\end{equation}
Moreover $y_n\to y$ in $\Z$, hence $|u_n|^{p^*}\mu_n$ is tight as a sequence of probabilities (recall $\Z$ is chosen proper). Thus, along a not relabeled subsequence, it converges in duality with $C_{b}(\Z)$ to some probability measure $\nu$. Additionally, up to a further subsequence, we have that $u_n$ is $L^{p^*}$-weak convergent to some $u \in L^{p^*}(\mu)$ with $\sup_n \|\nabla u_n\|_{L^p(\mu_n)}<\infty$ and also that $|\nabla u_n|^p\,\d\mu_n \weakto \omega$ {in duality with $C_{bs}(\Z)$} for some bounded Borel measure $\omega.$ We can invoke Lemma \ref{lem:pmGHW1ploc} and, up to a further subsequence, we also have that $u_n$ converges $L^{p}_{loc}$-strong to some $u\in L^p_{loc}(\mu)$, together with the facts $u \in W^{1,p}_{loc}(Y)$ and $|\nabla u|\in L^p(\mu)$.

Next, by Lemma \ref{lem:conccomp 2}, we infer the existence of countably many points $\{x_j\}_{j \in J}\subset Y$ and positive weights $(\nu_j),(\omega_j)\subset \R^+$, so that $\nu =|u|^{p^*}\mu +\sum_{j\in J}\nu_j\delta_{x_j}$ and $\omega \ge |\nabla u|^p\mu +\sum_{j\in J}\omega_j\delta_{x_j}$, with $A\omega_j \ge \nu_j^{p/p^*}$  and in particular $\sum_j \nu_j^{p/p^*}<\infty$.  Notice that $\lim_n \| \nabla   u_n \|^p_{L^p(\mu_n)} \ge \omega(\Z)$ holds by lower semicontinuity, therefore we can estimate
    \begin{align*}
        1 =  \lim_{n\uparrow\infty} \int |u_n|^{p^*}\, \d \mu_n  &\ge \liminf_{n\uparrow\infty}  \tilde  A_n\| \nabla   u_n \|^p_{L^p(\mu_n)} \ge  A \omega(\Z)\\
         &\ge  A\int |\nabla  u|^p\,\d \mu + \sum_{j\in J} \nu_j^{p/p^*} \overset{\eqref{eq:sob limit Y}}{\ge} \Big(\int |u|^{p^*}\, \d \mu\Big)^{p/p^*} + \sum_{j\in J} \nu_j^{p/p^*} \\
           &\ge \Big( \int |u|^{p^*}\, \d \mu  + \sum_{j\in J} \nu_j \Big)^{p/p^*} = \nu(Y)^{p/p^*} = 1,
    \end{align*}
    having used, in the last inequality, the concavity of the function $t^{p/p^*}$. In particular, all the inequalities must be equalities and, since $t^{p/p^*}$ is strictly concave, we infer that every term in the sum $\int |u|^{p^*}\, \d \mu  + \sum_{j\in J} \nu_j^{p/p^*}$ must vanish except one. By the assumption \eqref{eq:Levyscalings} and $|u|^{p^*}\mm_n \rightharpoonup \nu$ in $C_b(\Z),$ we have $\nu_j\le 1-\eta$ for every $j \in J$. Hence $\nu_j=0$ and $\|u\|_{L^{p^*}(\mu)}=1$.
    This means that $u_n$ converges $L^{p^*}$-strong to $u$. Moreover, retracing the equalities in the above we have that $\lim_n \int |\nabla u_n|^p\, \d \mu_n = \int |\nabla u|^p\,\d\mu$. This proves point ii). Finally, equality in the fourth inequality is precisely part iii) of the statement. The proof is now concluded.
\end{proof}

\section{Generalized existence}
In this part, we study generalized existence results for minimizers of the Sobolev inequality. We first handle the general case with nonnegative Ricci lower bounds and then the Alexandrov setting.
\subsection{Nonnegative Ricci lower bound}
\begin{theorem}\label{thm:general RCD}
    Let $(\X_n,\sfd_n,\mm_n)$ be a sequence of ${\sf RCD}(0,N)$-spaces for some $N\in (1,\infty)$ with ${\sf AVR}(\X_n) \in (V,V^{-1})$ for some $V>0$ and let $p \in (1,N)$. Set $p^*=pN/(N-p)$. Then, for every $0\neq u_n\in W^{1,p}_{loc}(\X_n)$ so that
    \[
       {\sf AVR}(\X_n)^{-\frac 1N}S_{N,p} -  \frac{\|u_n\|_{L^{p^*}(\mm_n)}}{\|\nabla u_n\|_{L^p(\mm_n)}}\to 0,\qquad\text{as }n\uparrow\infty,
    \]
    there is a not relabeled subsequence and $z_{n} \in\X_n,\sigma_{n}>0$ so that the following holds:
    \begin{itemize}
        \item[{\rm i)}] there exists a pointed ${\sf RCD}(0,N)$-space $(Y,\rho,\mu,z_0)$ so that, defining $(\X_{\sigma_{n}},\sfd_{\sigma_{n}},\mm_{\sigma_{n}},z_{n}) \coloneqq (\X_{n},\sigma_{n}\sfd_n,\sigma_{n}^N\mm_n,z_n)$, it holds $\X_{\sigma_{n}}  \overset{pmGH}{\to} Y$;
        \item[{\rm ii)}]  there is $0\neq u_\infty \in W^{1,p}_{loc}(Y)$ so that, defining $u_{\sigma_{n}} \coloneqq \frac{\sigma_{n}^{-N/p^*} u_{n}}{\|u_{n}\|_{L^{p^*}(\mm)}}  \in W^{1,p}_{loc}(\X_{\sigma_{n}})$, it holds
        \[
            u_{\sigma_{n}} \to u_\infty \quad \text{in }L^{p^*}\text{-strong},\qquad |\nabla u_{\sigma_{n}}| \to |\nabla u_\infty|\quad \text{in }L^p\text{-strong}.
        \]
        \item[{\rm iii)}] $(Y,\rho,\mu)$ is a metric measure cone with ${\sf AVR}(Y)=\lim_{n\uparrow\infty}{\sf AVR}(\X_n)$ ,  $z_0\in Y$ is a tip and
        \[
           \|u_\infty\|_{L^{p^*}(\mu)} = {\sf AVR}(Y)^{-\frac 1N}S_{N,p} \|\nabla u_\infty\|_{L^p(\mu)}.
        \]
        Moreover $u_\infty$ is a Euclidean bubble centered at $z_0$, i.e.\ there are $a\in\R,b>0$  so that
        \[
        u_\infty =  \frac{a}{(1+b\rho(\cdot,z_0)^{\frac{p}{p-1}})^\frac{N-p}{p}}.
        \]
    \end{itemize}
\end{theorem}
\begin{proof}
   By scaling, we can assume without loss of generality that $\|u_n\|_{L^{p^*}(\mm_n)}=1$ for all $n\in\N$. By approximation, we can further suppose that $u_n \in W^{1,p}(\X_n)$ (see, e.g, \cite[Lemma 3.2]{NobiliViolo24}). We denote $A_n \coloneqq {\sf AVR}(\X_n)^{-\frac 1N}S_{N,p}$ and observe that, by assumptions, $u_n$ satisfies
    \[
    \|u_n\|_{L^{p^*}(\mm_n)} \ge (A_n-1/n)\|\nabla u_n\|_{L^{p}(\mm_n)},\qquad\forall n \in\N.
    \]
    Consider now $\eta \in (0,1)$ given by $\eta= \eta_{N,p}/2$ for $\eta_{N,p}$ as in Theorem \ref{thm:CC extremal}, and take  $y_n \in \X_n$ and $t_n>0$ so that
    \[ 
        1-\eta = \int_{B_{t_n}(y_n)} |u_n|^{2^*}\, \d \mm_n =  \sup_{y \in \X_n} \int_{B_{t_n}(y)} |u_n|^{2^*}\, \d \mm_n ,\qquad \text{for each }n \in \N.
    \]
    Define accordingly $\sigma_n \coloneqq  t_n^{-1}$ and  $(Y_n,\rho_n,\mu_n,y_n) \coloneqq  (\X_n, \sigma_n\sfd_{\sigma_n},\sigma_n^N\mm_n,y_n)$. Clearly $(Y_n,\rho_n,\mu_n)$ is a sequence of ${\sf RCD}(0,N)$-spaces satisfying $\frac{\mu_n( B_R(y))}{\omega_N R^N} =\frac{ \mm_n( B_{ R/\sigma_n}(y))}{\omega_N (R/\sigma_n)^N}$ for all $R>0$ and $n\in\N$. In particular, we deduce 
    \[  
       {\sf AVR}(Y_n)  ={\sf AVR}(\X_n) ,\qquad  A_n \in (V^{\frac 1N}S_{N,p}, V^{-\frac 1N}S_{N,p}),
    \]
    and by scaling
    \[
    1-\eta =  \int_{B_1(y_n)} |u_{\sigma_n}|^{p^*}\, \d \mu_n\qquad \text{and}\qquad  \| u_{\sigma_n}\|_{L^{p^*}(\mu_n)}  \ge (A_n-1/n) \| \nabla u_{\sigma_n}\|_{L^p(\mu_n)}.
    \]
    for  all $n\in\N$. Consider then a not relabeled subsequence so that $A_n\to A$, for some $A>0$ finite so that, in particular, $\exists \lim_{n\uparrow\infty} {\sf AVR}(\X_n)$ along such subsequence.

    We are in position to invoke Theorem \ref{thm:CC extremal} and get points $y_n \in \X_n$ and scalings $\sigma_n>0$ so that, up to a subsequence, $(Y_n,\rho_n,\mu_n,y_n)$ pmGH-converges to some pointed  ${\sf RCD}(0,N)$ space $(Y,\rho,\mu, z)$ satisfying by stability of Sobolev constants (c.f.\ \cite[Lemma 4.1]{NobiliViolo21})
    \begin{equation}
    \|u\|_{L^{p^*}(\mu)} \le A \|\nabla u\|_{L^p(\mu)},\qquad \forall u \in W^{1,p}(Y).
    \label{eq:ausixliary}
    \end{equation}
    This shows conclusion i) with taking $z_n=y_n$ and $z_0=z$. 

    Thanks to the sharpness result in \cite[Theorem 4.6]{NobiliViolo21}, we directly deduce ${\sf AVR}(Y)^{-\frac 1N}S_{N,p} \le A$. Again by Theorem \ref{thm:CC extremal}, we know that $u_{\sigma_n}$ {converge} $L^{p^*}$-strong to some function $0\neq u_\infty \in W^{1,p}_{loc}(Y)$ with $\|\nabla u_{\sigma_n}\|_{L^p(\mm_{\sigma_n})}\to \|\nabla u_\infty\|_{L^p(\mu)}$ attaining equality in \eqref{eq:ausixliary}. This shows conclusion ii), recalling also the last conclusion of Lemma \ref{lem:pmGHW1ploc}. 
    
    We finally prove iii) and conclude the proof. We can estimate
    \begin{align*}
         {\sf AVR}(Y)^{-\frac 1N}S_{N,p} \| \nabla u_\infty \|_{L^p(\mu)} &\ge  \|u_\infty \|_{L^{p^*}(\mu)} =\lim_{n\uparrow\infty} \|u_{\sigma_n} \|_{L^{p^*}{(\mu_n)}} \\
         &\ge  \lim_{n\uparrow\infty} (A_n -1/n)\|\nabla u_{\sigma_n}\|_{L^p(\mu_n)} = A\|\nabla u_\infty\|_{L^p(\mu)},
    \end{align*}
    giving in turn
    \[ 
    {\sf AVR}(Y)^{-\frac 1d}S_{N,p} =A,\qquad \text{hence also}\qquad {\sf AVR}(Y)=\lim_{n\uparrow\infty}{\sf AVR}(\X_n).
    \]
    In particular, $u_\infty$ is a nonzero extremal function for the sharp Sobolev inequality on $Y$. Thanks to the rigidity result \cite[ii) in Theorem 1.5]{NobiliViolo24_PZ} we deduce that $Y$ is a cone and, for some tip  $z_0 \in Y$ and suitable $a \in \R, b>0$, we find $ u_\infty  =  a(1+b\sfd(\cdot,z_0)^{\frac{p}{p-1}})^\frac{p-N}{p}$. Then we can take $z_n \in Y_n$ any sequence such that $z_n\to z_0$ in $\Z$, since i) would be still satisfied replacing $z$ with $z_0$.
\end{proof}
The above applies also when $\X_n \equiv M$ is a fixed Riemannian manifold $(M,g)$ that is not isometric to the Euclidean space (or, more generally, for a fixed ${\sf RCD}(0,N)$-space that is not a cone). In particular, this result can be interpreted as a generalized existence result in the spirit of \cite{Lions84,Lions85} for minimizers of the Sobolev inequality on a fixed space. 

The next Theorem proves our first main result Theorem \ref{thm:qualitative Sob Manifold}. 
\begin{theorem}\label{thm:qualitative Sob}
For all $\eps>0, V\in(0,1), N>1$ and $p \in (1,N)$, there exists $\delta\coloneqq \delta(\eps,p,N,V)>0$ such that the following holds.
Let  $\Xdm$  be an $\RCD(0,N)$ space with ${\sf AVR}(\X) \in( V,V^{-1})$ and let $0\neq u\in W^{1,p}_{loc}(\X)$ be satisfying
\[
  \frac{\| u\|_{L^{p^*}(\mm)}}{\| \nabla u\|_{L^p(\mm)}} > {\sf AVR}(\X)^{-\frac 1N}S_{N,p}- \delta.
\]
Then, there are $a\in\R,b>0$ and $z_0 \in \X$ so that 
\begin{equation}\label{eq:main stability}
    \frac{\| \nabla(u - u_{a,b,z_0})\|_{L^{p}(\mm)}}{\| \nabla u\|_{L^{p}(\mm)}} \le \eps,
\end{equation}
where $u_{a,b,z_0}\coloneqq a(1+b\sfd(\cdot,z_0)^{\frac{p}{p-1}})^\frac{p-N}{p}$ and $|a|= c_{N,p}\|u\|_{L^{p^*}(\mm)}\avr(\X)^{-1} b^\frac{(p-1)(N-p)}{p^2},$ for some constant $c_{N,p}$ depending only on $N$ and $p$.
\end{theorem}
\begin{proof}
    If not, for any $n\in\N$ there are ${\sf RCD}(0,N)$ spaces $(\X_n,\sfd_n,\mm_n)$ with ${\sf AVR}(\X_n) \in (V,V^{-1})$ and there are $0\neq u_n \in W^{1,p}_{loc}(\X_n)\cap L^{p^*}(\mm_n)$ satisfying
    \begin{equation}
        \| u_n\|_{L^{p^*}(\mm_n)} \ge (A_n-1/n) \| \nabla u_n\|_{L^p(\mm_n)}, \label{eq:AVR sobolev proof}
    \end{equation}
    where $A_n\coloneqq  {\sf AVR}(\X_n)^{-\frac 1N}S_{N,p}$, but, by the contradiction hypothesis, { there is no constant $c_{N,p}$ (depending only on $N$ and $p$)} so that 
    \begin{equation}
        \inf_{a,b,z} \frac{\| \nabla ( u_n  -u_{a,b,z}  )\|_{L^p(\mm_n)}}{\|\nabla u_n\|_{L^p(\mm_n)}} > \eps,\qquad \forall n \in \N,\label{eq:contradiction AVR}
    \end{equation}
    { for all $b>0$ and $|a|= \|u\|_{L^{p^*}(\mm_n)}c_{N,p}^{-1}\avr(\X_n)^{-1} b^\frac{(p-1)(N-p)}{p^2}.$}

    By the generalized existence result in Theorem \ref{thm:general RCD}, we know that there is a not relabeled subsequence, scalings $\sigma_n>0$ and points $z_n \in \X_n$ so that $(\X_{\sigma_n},\sfd_{\sigma_n},\mm_{\sigma_n},z_n)$ pmGH-converges to a pointed metric measure cone $(Y,\rho,\mu,z_0)$, with tip $z_0$, and that $u_{\sigma_n}\coloneqq \sigma_n^{-N/p^*}u_n \in W^{1,p}_{loc}(Y_n)$ converges in $L^{p^*}$-strong to a Euclidean bubble $u_\infty = u_{a,b,z_0} \in W^{1,p}_{loc}(Y)\cap L^{p^*}(\mu)$.   { Up to changing the sign of $u_n$, we can clearly assume that $a>0.$ It  also is easy to see integrating in polar coordinates (c.f.\ \cite[Lemma 4.2]{NobiliViolo21}) that $\|u_{a,b,z_0}\|_{L^{p^*}(\mu)}=c_{N,p}^{-1}\avr(Y) a \cdot b^\frac{(1-p)(N-p)}{p^2},$ for some constant $c_{N,p}$ depending only on $N$ and $p.$ Hence we must have
    \begin{equation}\label{eq:formula ab}
      a =  c_{N,p}\avr(Y)^{-1} b^\frac{(p-1)(N-p)}{p^2}.
    \end{equation}
    } 
    Applying \cite[Lemma 7.2]{NobiliViolo24} twice, we get that $f\circ \rho_n(\cdot,z_n)$ converges $L^{p^*}$-strong to $u_{a,b,z_0}$ and that $ |\nabla(f\circ \rho_n(\cdot,z_n))| $ converges $L^p$-strong to $|\nabla u_{a,b,z_0}|$ where we defined $f(t) = a(1+bt^{\frac{p}{p-1}})^{\frac{p-N}{p}}$.  Recall that $u_{\sigma_n}$ and $|\nabla u_{\sigma_n}|$ also converge respectively in $L^{p^*}$ and $L^p$-strong to $u_{a,b,z_0}$.Therefore, by linearity of convergence (c.f.\  Lemma \ref{prop:linearity W1ploc convergence}, and since $p^*\ge p$ whence $L^p_{loc}$-strong convergence does hold) we get
    \[
        \lim_{n\uparrow\infty} \| \nabla (u_{\sigma_n} - f\circ\sfd_{\sigma_n}(\cdot,z_n))\|_{L^{p}(\mm_{\sigma_n})} =0.
    \]
    By scaling, we thus deduce that the sequence  $v_n \coloneqq   a\sigma_n^{d/p^*} (1+  b\sigma_n^\frac{p}{p-1}\sfd_n(\cdot,z_n)^\frac{p}{p-1})^{\frac{p-N}{p}}$ satisfies
    \[
       \limsup_{n\uparrow\infty} \frac{\| \nabla ( u_n - v_n )\|_{L^p(\mm_n)}}{\|\nabla u_n\|_{L^p(\mm_n)}} =0,
    \] 
    having also used that $\|\nabla u_n\|_{L^p(\mm_n)}\ge S_{N,p}^{-1}\avr(\X_n)^{1/N} \|u_n\|_{L^{p^*}(\X_n)}\ge S_{N,p}^{-1}V^{\frac 1N}.$ {Since $\avr(\X_n)\to \avr(Y)\ge V>0$ by iii) in Theorem \ref{thm:general RCD}, the same is true if we replace $v_n$ with $\tilde v_n\coloneqq \frac{\avr(Y)}{\avr(\X_n)}v_n $.  Since by \eqref{eq:formula ab} we have that $\tilde v_n = u_{a_n,b_n,z_n}$ for $b_n = b\sigma_n^2$, $a_n = c_{N,p}\avr(\X_n)^{-1} b_n^\frac{(p-1)(N-p)}{p^2},$ we find a contradiction with \eqref{eq:contradiction AVR}. The proof is therefore concluded.}
\end{proof}
\begin{remark}\label{rmk:non trivial generalization}
    Theorem \ref{thm:general RCD} and Theorem \ref{thm:qualitative Sob} generalize previous results obtained in \cite{NobiliViolo24} for all $p \in(1,\infty)$. For this generalization, there are two crucial ingredients that, at the time of \cite{NobiliViolo24}, where not available. Indeed, Theorem \ref{thm:general RCD} builds upon a nontrivial adaptations of concentration compactness methods and stability of Sobolev functions along converging ${\sf RCD}$ spaces to general exponent $p$ (c.f.\ Theorem \ref{thm:CC extremal} and Section \ref{sec:Mosco}). Finally, to show iii) in Theorem \ref{thm:general RCD} we relied on a recent characterization of Sobolev minimizers deduced in \cite[Theorem 1.5]{NobiliViolo24_PZ} and working for possibly $p\neq 2$ (differently from \cite[Theorem 5.3]{NobiliViolo24}, see also \cite[Remark 2.9]{Nobili24_overview} for details). This in turn was obtained by carefully revisiting the equality case in the P\'olya-Szeg\H{o} inequality in this setting (c.f.\ \cite[Theorem 1.3]{NobiliViolo24_PZ}). \fr
\end{remark}
%
{In the next result we show that, provided the measure of small balls is big enough uniformly, then almost extremal functions must be diffused.  Note that, while on  a Riemannian manifold  with non-negative Ricci curvature and Euclidean volume growth (different from $\rr^n$) \eqref{eq:eps not a cone} is true at every point for some $\rho$, it is not clear if it is true uniformly. We will show in the next section that non-negative sectional curvature is enough. }
\begin{theorem}[Refined stability]\label{thm:refined stability}
    Under the assumptions and notations of Theorem \ref{thm:qualitative Sob} suppose also that \footnote{Inequality \eqref{eq:eps not a cone} is equivalent to  $ \inf_{x\in \X} \frac{\mm(B_\rho(x))}{\omega_N \rho^{N}}\ge \avr(\X)+\delta,$ for some $\rho>0$ and $\delta>0$, by Bishop-Gromov monotonicity.}
    \begin{equation}\label{eq:eps not a cone}
        \inf_{x\in \X} \frac{\mm(B_\rho(x))}{\omega_N \rho^{N}}\ge \avr(\X)+\rho,
    \end{equation}
    holds for some $\rho>0$.
    Then, letting $\delta$ {depend} also on $\rho,$ for every $\eps$  the same conclusion of Theorem \ref{thm:qualitative Sob} holds with the additional property that
    \begin{equation}\label{eq:refined stability part 1}
         |u_{a,b,z_0}|\le \|u\|_{L^{p^*}(\mm)}\eps,   \quad \text{(or equivalently $b<\eps$).}
    \end{equation}
\end{theorem}
\begin{proof}
{ The equivalence of the two statements in \eqref{eq:refined stability part 1} (up to decreasing $\delta$) follows immediately from the fact that, by Theorem \ref{thm:qualitative Sob} we can take $|a|= c_{N,p}\|u\|_{L^{p^*}(\mm)}\avr(\X)^{-1} b^\frac{(p-1)(N-p)}{p^2},$ and in particular
$$|u_{a,b,z_0}|\le |u_{a,b,z_0}(z_0)|=a.$$
}
 We now argue by contradiction. Suppose the statement is false. Then there exist  $\eps>0$, $\rho$, a sequence $\delta_n\to 0$,  $\RCD(0,N)$-spaces  $(\X_n,\sfd_n,\mm_n)$  with ${\sf AVR}(\X_n)\ge V$ satisfying \eqref{eq:eps not a cone} and functions $u_n\in W^{1,p}_{loc}(\X_n)$ such that 
  \[
    \frac{\| u_n\|_{L^{p^*}(\mm_n)}}{\| \nabla u_n\|_{L^p(\mm_n)}} - {\sf AVR}(\X_n)^{-\frac 1N}S_{N,p}\to 0.
  \]
  and for every choice of $a,b,z_0$ such that 
  $  \| \nabla (u_n - u_{a,b,z_0})\|_{L^{p}(\mm_n)} \le \eps \| \nabla u\|_{L^{p}(\mm_n)}$ it holds that   $b\ge \eps.$
Arguing as in the proof of Theorem \ref{thm:qualitative Sob} and up to passing to a subsequence, we can find numbers $\sigma_n>0$ and points $z_n\in \X_n$ such that the spaces $(Y_n,\rho_n,\mu_n,z_n)\coloneqq (\X_n,\sigma_n\sfd_n,\sigma_N^{-N}\mm_N,z_n)$ pmGH-converge to a $\RCD(0,N)$ cone $(Y,\rho,\mu,z_0)$, with tip $z_0,$ $\avr(Y)=\lim_n \avr(\X_n)$ and the functions   $v_n \coloneqq   a_n (1+ b_n\sfd_n(\cdot,z_n)^\frac{p}{p-1})^{\frac{p-N}{p}}$
satisfy
\[
       \limsup_{n\uparrow\infty} \frac{\| \nabla ( u_n - v_n )\|_{L^p(\mm_n)}}{\|\nabla u_n\|_{L^p(\mm_n)}} =0,
    \] 
    where $b_n= b\sigma_n^\frac{p}{p-1}$ and $a_n= a\sigma_n^\frac{n-p}{p}$ for some $a>0,$ $b>0$. By assumption we must have $b_n\ge \eps$ and so $\sigma_n\ge c>0$ for some constant $c>0$ independent of $n.$ Hence by Bishop-Gromov monotonicity
    \begin{align*}
         \lim_{n\uparrow\infty} \avr(\X_n)&=\avr(Y)=\frac{\mu(B_{c\rho}(z))}{\omega_N(c\rho)^N}=\lim_{n\uparrow\infty}\frac{\mu_n(B^{Y_n}_{c\rho}(z_n))}{\omega_N(c\rho)^N}\\
         &=\lim_{n\uparrow\infty} \frac{\mm_n(B_{c\rho/\sigma_n}(z_n))}{\omega_N(c\rho/\sigma_n)^N}\ge \lim_{n\uparrow\infty} \frac{\mm_n(B_{\rho}(z_n))}{\omega_N\rho^N}\ge \lim_{n\uparrow\infty}\avr(\X_n)+\rho,
    \end{align*}
    which is a contradiction.
\end{proof}

\begin{remark}
    Condition \eqref{eq:eps not a cone} is sharp in the following sense. If there {exist sequences} $x_i \in \X$ and $r_i\to 0^+$ such that 
    \begin{equation}\label{eq:almost conical small balls}
         \frac{\mm(B_{r_i}(x_i))}{\omega_N {r_i}^{N}}\le \avr(\X)+r_i,
    \end{equation}
    then for any $b>0$ the functions $u_i= (1+b\sfd(\cdot,x_i)^{\frac{p}{p-1}})^\frac{p-N}{p}$  are extremizing for the Sobolev inequality in $\X$. Indeed  a suitable subsequence  of $(\X,\sfd,\mm,x_i)$ would pmGH-converge to a limit space $(\X_\infty,\sfd_\infty,\mm_\infty,x_\infty)$. Moreover \eqref{eq:almost conical small balls} implies that $\X_\infty$ must be a cone with  $\avr(\X_\infty)=\avr(\X)$. Then any function $u= (1+b\sfd(\cdot,x_\infty)^{\frac{p}{p-1}})^\frac{p-N}{p}$ is {extremal} for the Sobolev inequality in $\X_\infty.$ (c.f.\ \cite[ii) in Theorem 1.5]{NobiliViolo24_PZ}). From this the fact that $u_i$ is extremizing follows from \cite[Lemma 7.2]{NobiliViolo24}. \fr
\end{remark}
\subsection{Nonnegative sectional lower bound}
In this part we explore refined stability results in Alexandrov spaces. Roughly, the following asserts that on a fixed ${\sf CBB}(0)$ space with Euclidean volume growth and which is not a cone, extremizing sequences for the Sobolev inequality have a diffused asymptotic behavior up to the isometries of the space. As a corollary, we also prove our main result Thereom \ref{thm: stability Sect manifold} on manifolds with nonnegative sectional curvature lower bounds.
\begin{theorem}\label{thm:stability CBB}
    Under the same assumptions and notations of Theorem \ref{thm:qualitative Sob}, suppose also that $(\X,\sfd)$ is an $N$-dimensional ${\sf CBB}(0)$ space with $N\in\N$ and that $(\X,\sfd,\HH^N)$ is not a metric measure cone. 
      Write also $\X=\rr^k\times Y$ for some $0\le k<N$ with $Y$ that does not split off any line and fix $y_0 \in Y$.   Then,  the same conclusion of Theorem \ref{thm:qualitative Sob} holds (letting $\delta$ depend also on $y_0$)   with the additional properties that
    \begin{equation}
        z_0\in \rr^k\times\{y_0\} \qquad\text{and}\qquad  b<\eps, \label{eq:up to isometries}
    \end{equation}
    (and in particular $  |u_{a,b,z_0}|\le \|u\|_{L^{p^*}(\mm)}\eps$).
\end{theorem}
\begin{proof}
   We start by showing that \eqref{eq:eps not a cone} holds for some $\rho(\X)>0.$ Suppose the contrary, i.e.\ that there exist a sequence of points $x_i\in \X$ and a sequence of radii $r_i\to 0$ such that
    \begin{equation}
         \frac{\mm(B_{r_i}(x_i))}{\omega_N {r_i}^{N}}\le \avr(\X)+r_i.
    \end{equation}
   We claim that $x_i$ must be a diverging sequence. If this is not the case, up to passing to a subsequence, it converges to some $\bar x\in \X.$  For any $r>0$, Bishop-Gromov monotonicity implies
    \[
     \frac{\mm(B_{r}(\bar x))}{\omega_N {r}^{N}}=\lim_{i\uparrow \infty} \frac{\mm(B_{r}(x_i))}{\omega_N {r}^{N}}\le \liminf_{i\uparrow \infty} \frac{\mm(B_{r_i}(x_i))}{\omega_N {r_i}^{N}}\le\avr(\X).
    \]
    Hence $ \lim_{r\to 0^+}\frac{\mm(B_{r}(\bar x))}{\omega_N {r}^{N}}\le\avr(\X)$, which implies that $\X$ is a cone with tip $\bar x$, which is a contradiction.  Thus $x_i$ is a diverging sequence.
    Up to passing to a subsequence we have that  $(\X,\sfd,\HH^N,x_i)$ pmGH-converges to a limit space $(\X_\infty,\sfd_\infty,\mm_\infty,x_\infty)$.
    By combining \cite[Lemma 4.2]{AntonelliBrueFogagnoloPozzetta22} with Theorem \ref{thm:asymptotic cone CBB}, we must have that $\avr(\X_\infty)\ge \avr(\X)+\delta$ for some $\delta>0.$ Hence  $ \frac{\mm_\infty(B_{1}(x_\infty))}{\omega_N }\ge \avr(\X)+\delta,$ however by pmGH-convergence 
    \[
    \frac{\mm_\infty(B_{1}(x_\infty))}{\omega_N }\le \liminf_{i\uparrow \infty}  \frac{\mm(B_{1}(x_i))}{\omega_N }\le \liminf_{i\uparrow \infty} \frac{\mm(B_{r_i}(x_i))}{\omega_N {r_i}^{N}}\le\avr(\X),
    \]
    which is again a contradiction. We have thus shown that \eqref{eq:eps not a cone} holds.

  We now pass to the proof of the statement  by contradiction. Fix $y_0\in Y.$ Suppose that there exist $\eps>0$, a sequence $\delta_n\to 0$ and functions $u_n\in W^{1,p}_{loc}(\X)$ such that 
  \begin{equation}
        \frac{\| u_n\|_{L^{p^*}(\HH^N)}}{\| \nabla u_n\|_{L^p(\HH^N)}} - {\sf AVR}(\X)^{-\frac 1N}S_{N,p}\to 0,
  \end{equation}
  and for every choice of $a,b,z$ such that 
  $ \| \nabla (u_n - u_{a,b,z})\|_{L^{p}(\HH^N)} \le \eps\| \nabla u\|_{L^{p}(\HH^N)}$,
  either
  \begin{equation}\label{eq:contradiction in CBB}
        b\left(1+ {\rm dist}\big(\rr^k\times \{y_0\},z\big)^\frac{p}{p-1}\right)\ge \eps,
  \end{equation}
  or 
  \begin{equation}\label{eq:contradiction in CBB 2}
      z\notin \rr^k\times \{y_0\}
  \end{equation}
  hold.
  By the same argument in the proof of Theorem \ref{thm:qualitative Sob}  by taking $\X_n=\X$ (which we can repeat because we showed that {the assumption} \eqref{eq:eps not a cone} is true for some $\rho>0$)  we can find a sequence $\sigma_n\to 0^+$ and points $z_n\in \X$ such that the spaces $(Y_n,\rho_n,\mu_n,z_n)\coloneqq (\X_n,\sigma_n\sfd,\sigma_n^N\HH^N,z_n)$, where $\HH^N$ is the Hausdorff measure with respect to $\sfd,$ pmGH-converge to a $\RCD(0,N)$ cone $(Y,\rho,\mu,z_0)$ with tip $z_0 \in Y$ and $\avr(Y)= \avr(\X)$. Moreover the functions   $v_n \coloneqq   a_n (1+ b_n\sfd(\cdot,z_n)^\frac{p}{p-1})^{\frac{p-N}{p}}$
satisfy
\begin{equation}\label{eq:ananas}
       \limsup_{n\uparrow\infty} \frac{\| \nabla ( u_n - v_n )\|_{L^p(\HH^N)}}{\|\nabla u_n\|_{L^p(\HH^N)}} =0,
\end{equation}
    where $a_n = a \sigma_n^{d/p^*}$ and $b_n= b\sigma_n^\frac{p}{p-1}$ for some $a\in\R,b>0$. Suppose that \eqref{eq:contradiction in CBB} holds for infinitely many $n$ and so for all of them, up to pass to a subsequence. Since $\sigma_n\to 0$ we must have (up to a further subsequence)
    \[
    \sigma_n{\rm dist}\big(\rr^k\times \{y_0\},z_n\big)\ge (\eps/(2b))^\frac{p-1}{p}\eqqcolon\delta>0.
    \]
Up to isometries we can assume that $z_n=(0,x_n)$ for some $x_n\in Y$ and so $ \sigma_n\sfd(z_n,(0,y_0))\ge \eps$.
In particular, since $\sigma_n\to 0$, we have $r_n\coloneqq \sfd(z_n,(0,y_0))\to +\infty.$ Up to passing to a subsequence   $ (\X,r_n^{-1}\sfd,\HH^N/r_n^N,(0,y_0))$  converge to an asymptotic cone  $(\X_\infty,\sfd_\infty,\mm_\infty,y_\infty)$, where $\HH^N$ is with respect to $\sfd$. Combining \cite[ii) in Lemma 4.2]{AntonelliBrueFogagnoloPozzetta22} with Theorem \ref{thm:asymptotic cone CBB}, we have that  for all $z\in \X_\infty$ with $\sfd_\infty(z,y_\infty)=1$ it holds
\[
\lim_{r\to 0^+} \frac{\mm_\infty(B_r(z))}{r^N\omega_N}\ge  \avr(\X)+\delta
\]
for some $\delta>0$. By scaling, we note that $z_n\to z_\infty$ for some $z_\infty \in \X_\infty$ with $\sfd_\infty(z_\infty,y_\infty)=1$. Additionally by measure convergence
\begin{align*}
    \frac{\mm_\infty(B_r(z_\infty))}{r^N\omega_N}&\le \liminf_{n\uparrow\infty} \frac{\HH^N(B_{r r_n}(z_n))}{(r r_n)^N\omega_N}
=\liminf_{n\uparrow\infty} \frac{\HH^N(B_{r \sigma_nr_n/\sigma_n}(z_n))}{(r r_n\sigma_n/\sigma_n)^N\omega_N}\le \liminf_{n\uparrow\infty} \frac{\HH^N(B_{r \delta/\sigma_n}(z_n))}{(r \delta/\sigma_n)^N\omega_N}\\
&=\liminf_{n\uparrow\infty} \frac{\sigma_n^N\HH^N(B_{r \delta/\sigma_n}(z_n))}{(r \delta)^N\omega_N}=\frac{\mu(B^Y_{r\delta}(z_0))}{(r\eps)^N\omega_N},
\end{align*}
for all $r>0$.  Sending  $r\to 0$ we deduce that $\avr(\X)+\delta\le \avr(\X)$ which is a contradiction. Note that in fact we showed that $ \sigma_n{\rm dist}\big(\rr^k\times \{y_0\},z_n\big)\to 0.$ 

All in all, for $u_n$ we have showed that \eqref{eq:contradiction in CBB 2} must hold with the choice $a_n,b_n,z_n$ for infinitely many $n.$ However, since  $ \sigma_n{\rm dist}\big(\rr^k\times \{y_0\},z_n\big)\to 0$, the spaces $(\X_n,\sigma_n\sfd,\sigma_n^N\HH^N,(0,y_0))$, still pmGH-converge to the same cone $(Y,\rho,\mu,z_0)$ with the same tip $z_0$. Therefore we deduce that the functions  $v_n \coloneqq   a_n (1+ b_n\sfd(\cdot,(0,y_0))^\frac{p}{p-1})^{\frac{p-N}{p}}$  still verify \eqref{eq:ananas} (as we did in the very end of the proof of Theorem \ref{thm:general RCD}). In particular, the choice $a_n,b_n,(0,y_0)$  is then admissible, and does not satisfy neither \eqref{eq:contradiction in CBB} nor \eqref{eq:contradiction in CBB 2}. This gives a contradiction and concludes the proof.
\end{proof}
\begin{remark}\label{rmk:counter}
    \rm We remark that we cannot expect the first in \eqref{eq:up to isometries} to hold if $\X$ is only assumed an ${\sf RCD}(0,N)$ space with Euclidean volume growth and which is not a cone. Indeed, a counterexample is given by a carefully chosen warped metric $g$ in the four dimensional manifold $M=\R\times [0,+\infty)\times \S^2$ studied in \cite[Pag. 913-914]{KasueWashio90} with the following properties:
    \[
        {\sf Ric}_g\ge 0,\qquad  {\sf AVR}(M)>0,\qquad (M,g)\text{ does not split isometrically a line},
    \]
    but such that
    \[
        M\ni (s,p)\mapsto (s+t,p)\in M \qquad \text{is an isometry for any $t>0$}.
    \]
    {  Notice that, if $(u_n )\subset W^{1,p}_{loc}(M)$ is \emph{any} extremizing sequence for the Sobolev inequality in $(M,g)$, then also $\tilde u_n\coloneqq u_n ( \cdot -t_n, \cdot)$ is extremizing for every sequence $|t_n|\uparrow \infty$. However, for any fixed point $z_0=y_0 \in M$ and numbers $a\in \rr$, $b>0,$ it {clearly} holds that
    \[
    \liminf_{n\to +\infty} \frac{\|\nabla ( \tilde u_n- u_{a,b,z_0})\|_{L^p(M)}}{\|\nabla \tilde u_n\|_{L^p(M)}}\ge  1.
    \]}
    \fr
\end{remark}

\medskip 
\noindent\textbf{Acknowledgments}. F.N. is  a member of INDAM-GNAMPA. This work was supported by the European Union (ERC, ConFine, 101078057) and by the  MIUR Excellence Department Project awarded to the Department of Mathematics, University of Pisa, CUP I57G22000700001. Both authors would like to thank F. Glaudo for bringing to our attention a question related to Theorem \ref{thm: stability Sect manifold}.

\medskip

\noindent\textbf{Data availability}. No data was generated or used during the study.

\def\cprime{$'$} \def\cprime{$'$}

\end{document}